\documentclass[a4paper,11pt]{amsart}
\usepackage{amsmath}
\usepackage{amsthm,amsfonts,amssymb,graphics}
\usepackage{pgfplots}
\usepackage[foot]{amsaddr}

\usepackage{hyperref}
\usepackage[normalem]{ulem}
\usepackage{enumerate}
\usepackage{geometry}
\usepackage{dsfont}
\usepackage{centernot}
\usepackage{bbm}

\newtheorem{theorem}{Theorem}[section]
\newtheorem{lemma}[theorem]{Lemma}
\newtheorem{proposition}[theorem]{Proposition}
\newtheorem{corollary}[theorem]{Corollary}
\newtheorem{assumption}[theorem]{Assumption}

\newtheorem{example}[theorem]{Example}

\theoremstyle{remark}
\newtheorem{remark}[theorem]{Remark}

\newtheorem{definition}[theorem]{Definition}

\DeclareMathOperator*{\doublepath}{\Longleftrightarrow}

\newcommand\Wc{\mathcal{W}}
\newcommand\Uc{\mathcal{U}}
\numberwithin{equation}{section}

\newcommand \id{\mathds 1}
\newcommand {\R} {\mathbb{R}}
\newcommand {\E} {\mathbb{E}}
\newcommand {\N} {\mathbb{N}}
\newcommand {\Z} {\mathbb{Z}}

\newcommand {\X} {\mathbb{X}}
\renewcommand {\P} {\mathbb{P}}
\newcommand{\QQ} {\mathbb{Q}}
\newcommand {\rev} {\textrm{Rev}}
\newcommand {\Var} {\textrm{Var}}
\newcommand {\Supp} {\textrm{Supp}}
\newcommand {\Infl} {\textrm{Infl}}

\newcommand {\Ber} {\textrm{Ber}}
\newcommand {\cov} {\textrm{Cov}}
\newcommand {\one} {\mathbbm{1}}
\newcommand\compl{\mathsf{c}}

\begin{document}
\title[Upper bounds on the one-arm exponent]{Upper bounds on the one-arm exponent for \\ dependent percolation models}
\author{Vivek Dewan$^1$}
\email{vivek.dewan@univ-grenoble-alpes.fr}
\address{$^1$Institut Fourier, Universit\'{e} Grenoble Alpes}
\author{Stephen Muirhead$^2$}
\email{smui@unimelb.edu.au}
\address{$^2$School of Mathematics and Statistics, University of Melbourne}
\begin{abstract}
We prove upper bounds on the one-arm exponent $\eta_1$ for a class of dependent percolation models which generalise Bernoulli percolation; while our main interest is level set percolation of Gaussian fields, the arguments apply to other models in the Bernoulli percolation universality class, including Poisson-Voronoi and Poisson-Boolean percolation. More precisely, in dimension $d=2$ we prove that $\eta_1 \le 1/3$ for continuous Gaussian fields with rapid correlation decay (e.g.\ the Bargmann-Fock field), and in $d \ge 3$ we prove $\eta_1 \le d/3$ for finite-range fields, both discrete and continuous, and $\eta_1 \le d-2$ for fields with rapid correlation decay. Although these results are classical for Bernoulli percolation (indeed they are best-known in general), existing proofs do not extend to dependent percolation models, and we develop a new approach based on exploration and relative entropy arguments. The proof also makes use of a new Russo-type inequality for Gaussian fields, which we apply to prove the sharpness of the phase transition and the mean-field bound for finite-range fields.
\end{abstract}
\date{\today}
\thanks{}

\maketitle

\section{Introduction}

The critical phase of percolation models is believed to be described by a set of \textit{critical exponents} which govern the power-law behaviour of macroscopic observables at, or near, criticality. These exponents have been thoroughly investigated for Bernoulli percolation (see, e.g., \cite[Chapter 9]{gri99}) and there is a general expectation that, by \textit{universality}, they are identical within the class of dependent percolation models with sufficient decay of correlations. However not much is rigorously known about critical exponents for dependent percolation models.

In this paper we study the \textit{one-arm exponent} governing the distribution of critical clusters with large diameter. Using a new approach we establish upper bounds on the exponent for a wide class of dependent percolation models induced by the excursion sets of Gaussian fields (`Gaussian percolation'), which contains Bernoulli site percolation as a special case; roughly-speaking our results extend the state-of-the-art for Bernoulli percolation to this general class.

Although we focus on Gaussian percolation, our approach could be adapted to other models in the same universality class, including Poisson-Voronoi and Poisson-Boolean percolation.

\subsection{Gaussian percolation and the one-arm exponent}
We consider the following class of dependent percolation models defined on either $\R^d$ or $\Z^d$, $d \ge 2$. Let $f$ be either (i) a continuous stationary-ergodic centred Gaussian field on $\mathbb{R}^d$, or (ii) a stationary-ergodic centred Gaussian field on $\Z^d$. We write $\X \in \{\R, \Z\}$ depending on the domain $\X^d$ of the field.

For $\ell \in \mathbb{R}$ write $\mathbb{P}_\ell[\cdot]$ to denote the law of $f + \ell$ (abbreviating $\P = \P_0$). Then the excursion sets $\{ f   + \ell \ge 0 \} := \{ x  \in \X^d : f(x) + \ell \ge 0\}$ induce a stationary-ergodic percolation model on $\X^d$ via the connectivity relation
\[  \{ A \longleftrightarrow B \} := \{ \text{there exists a path in $\X^d$ in $\{f  \ge 0\}$ that intersects $A$ and $B$}\} \]
for closed sets $A, B \subset \X^d$. When $f$ is an i.i.d.\ field on $\Z^d$, this model is equivalent to Bernoulli \textit{site} percolation on $\Z^d$. When $f$ is a continuous field on $\R^d$, $d=2$, the model is \textit{self-dual} at $\ell = 0$, and in that sense is similar to \textit{bond} percolation on $\Z^2$.

 Let $\Lambda_R := [-R, R]^d \cap \X^d$. By monotonicity we may define the \textit{critical level}  $\ell_c = \ell_c(f) \in [-\infty, \infty]$ of the model to be the unique level satisfying 
\[ \theta(\ell) := \mathbb{P}_\ell[ \Lambda_1 \longleftrightarrow \infty  ] :=  \lim_{R \to \infty} \mathbb{P}_\ell[ \Lambda_1  \longleftrightarrow \partial \Lambda_R  ] = \begin{cases} 0 &  \text{if } \ell < \ell_c , \\ > 0 & \text{if } \ell > \ell_c . \end{cases} \]
By self-duality one expects that $\ell_c = 0$ if $f$ is continuous and $d=2$, whereas in general one expects $\ell_c \in (-\infty, \infty)$ (see \cite{rv20, mv20, gv19, riv19, mrv20} and \cite{ms83a, ms83b} respectively for sufficient conditions, which are implied by Assumption~\ref{a:gf} below). It is believed that $\theta(\ell_c) = 0$, although so far this is only known if $d=2$ and $f$ is positively-correlated \cite{har60,gkr88,alex96}, or if $f$ is i.i.d.\ and $d \ge 11$ \cite{fh17}.

For Bernoulli percolation \cite{cc86,ab87}, as well as certain other (non-Gaussian) dependent percolation models \cite{dcrt19a,dcrt19b,dcrt20}, it is known that $\theta(\ell)$ satisfies the \textit{mean-field bound} 
\begin{equation}
\label{a:mflb}
\tag{MFB}
\theta(\ell) \ge c (\ell - \ell_c) 
\end{equation}
for a constant $c = c(d)$ and $\ell > \ell_c$ sufficiently close to $\ell_c$; this bound is expected to be tight for dimensions $d \ge d_c = 6$ in which critical exponents take their mean-field value. We expect \eqref{a:mflb} to be true in general for Gaussian percolation but this has not yet been established; in this paper we prove it for Gaussian models with \textit{finite-range dependence}.

At criticality $\ell = \ell_c$ it is believed that connection probabilities obey a power law, in the sense that there exists a \textit{one-arm exponent} $\eta_1 > 0$ such that, as $R \to \infty$ and for $r = o(R)$,
\begin{equation}
\label{e:eta1heu}
    \mathbb{P}_{\ell_c}[ \Lambda_r \longleftrightarrow \partial \Lambda_R   ]  = (r/R)^{\eta_1 + o(1)} .  
    \end{equation}
While the existence of $\eta_1$ is not known (except in the i.i.d.\ case in high dimension), since we are interested in upper bounds we define
\begin{equation}
\label{e:eta}
\eta_1 := \liminf_{R \to \infty}  \frac{- \log \mathbb{P}_{\ell_c}[ \Lambda_1 \longleftrightarrow \partial \Lambda_R   ] }{ \log R}  .
\end{equation}
Clearly upper bounds on \eqref{e:eta} imply upper bounds on the exponent in \eqref{e:eta1heu} assuming its existence. Note however that the choice of $\liminf$, rather than $\limsup$, in the definition of $\eta_1$ is deliberate and yields an a priori weaker bound (c.f.\ Remark \ref{r:limsup}).

For Bernoulli percolation it is expected that
\[  \eta_1 = \begin{cases} & \! \! \! \! 5/48   \qquad \ \  \hspace{0.05cm} d = 2, \qquad \qquad 0.48\ldots  \qquad   d = 3, \qquad  \qquad  0.95\ldots \qquad d=4 , \\ & \! \! \! \! 1.5\ldots  \qquad d = 5 , \qquad \qquad 2  \qquad  \qquad  \hspace{0.3cm} d \ge d_c =  6 . \end{cases} \]
According to the \textit{Harris criterion} (see \cite{wei84}, or \cite{bmr20} for further discussion), the value of $\eta_1$ should be the same for all Gaussian models whose correlations decay faster than $|x|^{-2/\nu }$ as $x \to \infty$, where $\nu = \nu(d) >0$ is the \textit{correlation length exponent} of Bernoulli percolation (see Section~\ref{s:ce} for a definition). For models with stronger correlations, the value of $\eta_1$ may be different \cite{wei84,dpr21}.

\subsection{Upper bounds on the one-arm exponent}
\label{s:upperbounds}

Recall that $f$ is a stationary-ergodic centred Gaussian field on the state space $\X^d$, $d \ge 2$, where $\X \in \{\R, \Z\}$. We refer to $\X = \R$ and $d=2$ as the \textit{self-dual case}.

We will always assume that $f$ has a \textit{spatial moving average} representation $f = q \star W$, where $q \in L^2(\X^d) \neq 0$ is Hermitian (i.e.\ $q(x) = q(-x)$), $W$ is the white noise on~$\X^d$ (interpreted as a collection of i.i.d.\ standard Gaussian random variables indexed by $\Z^d$ if $\X = \Z$), and $\star$ denotes convolution. Such a representation always exists if the covariance kernel $K(\cdot) := \mathbb{E}[f(0) f(\cdot)] = (q \star q)(\cdot)$ is absolutely integrable, since then we may define $q :=\mathcal{F}[ \sqrt{ \rho } ]$, where $\mathcal{F}$ denotes the Fourier transform and $\rho = \mathcal{F}[K]$ is the \textit{spectral density} of the field. Bernoulli site percolation corresponds to the case $\X = \Z$ and $q(x) = \id_{x = 0}$.

We will further assume that $q$ satisfies the following basic properties:

\begin{assumption}[Basic assumptions]
\label{a:gf}
$\,$
\begin{enumerate}[(a)]
\item (Decay of correlations, with parameter $\beta > d$) There exists a $c > 0$ such that, for all $x \in \mathbb{R}^d$,
\begin{equation}
    \label{e:qdecay}
  |q(x)| \le  c |x|^{-\beta } .  
  \end{equation}
\item (Symmetry) $q$ is symmetric under negation and permutation of the coordinate axes.
\item (Regularity, only if $\X = \R$) $q \in C^3(\R^d)$, all derivatives up to third-order are in $L^2(\R^d)$, and \eqref{e:qdecay} holds with $|\nabla q(x)|$ in place of $|q(x)|$.
\end{enumerate}
\end{assumption}

Let us make some remarks on Assumption \ref{a:gf}. The decay condition ensures that $K = q \star q \in L^1(\X^d)$, and so correlations decay polynomially with exponent $\beta > d$. This also ensures that the spectral density $\rho = \mathcal{F}[K]$ is continuous. The symmetry assumption is crucial for our results in self-dual case (for instance to prove RSW estimates), but it also simplifies some aspects of the proof in general. In the case $\X = \R$, the regularity condition ensures that $f$ is $C^2$-smooth almost surely. Along with the continuity of the spectral measure, this also implies that $(f, \nabla f, \nabla^2 f)$ is non-degenerate (i.e.\ its evaluation on a finite number of distinct points is a non-degenerate Gaussian vector, see \cite[Lemma A.2]{bmm20}). Finally, as we mentioned above,  Assumption~\ref{a:gf} is sufficient to prove that $\ell_c \in (-\infty,\infty)$ \cite{ms83a,ms83b}, and that $\ell_c = 0$  in the self-dual case (see \cite[Theorem 1.3]{mrv20} and Remark~1.9 therein).

For most of our results we also assume that
\begin{equation}
\label{a:pos1}
\tag{POS}
\int q := \int_{\X^d} q(x) \, dx > 0 ,
\end{equation}
where for a function $g:\Z^d \to \R$ we interpret $\int_{D} g(x) \, dx$ as $\sum_{x \in D \cap \Z^d} g(x)$. This is equivalent to the spectral density $\rho$ being strictly positive at the origin, and is a natural assumption when studying how properties of a Gaussian field change with the level; see e.g.\ \cite{mv20, bmm19}. 

For some of our results we further assume that $f$ is \textit{positively correlated}:
\begin{equation}
\label{a:pos2}
\tag{PA}
\text{$K(x) = (q \star q)(x) \ge 0$ for all $x \in \X^d$.}
\end{equation}
This is equivalent to the \textit{FKG inequality} holding for the field (i.e.\ the field is \textit{positively associated}), so that events increasing with respect to the field are positively correlated \cite{pi82}.\footnote{Although in \cite{pi82} this is proven only for finite Gaussian \textit{vectors}, one can deduce positive associations for all increasing events considered in this paper via standard approximation arguments, see \cite[Lemma A.12]{rv19}.} Note that \eqref{a:pos2} is stronger than \eqref{a:pos1}, since the former implies that $\rho = \mathcal{F}[K]$ is positive definite, and so strictly positive at the origin.

\smallskip
We are now ready to state our results. Our first result concerns Gaussian percolation with finite-range of dependence:

\begin{theorem}[Finite-range models]
\label{t:bou}
Suppose $f$ satisfies Assumption \ref{a:gf}, \eqref{a:pos1}, and $q$ has bounded support. Then $\eta_1 \le d/3$. In the self-dual case ($f$ is continuous and $d=2$), if additionally \eqref{a:pos2} holds, then there exists $c > 0$ such that, for all $R \ge 1$,
\[ \mathbb{P}_{p_c}[ \Lambda_1 \longleftrightarrow \partial \Lambda_R] \ge c R^{-1/3} , \]
and in particular $\eta_1 \le 1/3$.
\end{theorem}

\begin{remark}
\label{r:ber}
Except in some very special cases, the bounds in Theorem \ref{t:bou} match the best-known bounds for Bernoulli percolation:
\begin{enumerate}
\item For Bernoulli percolation on $\Z^d$ in general dimension $d \ge 3$, the best known bound is $\eta_1 \le d/3$~\cite{bcks99}, deduced by combining the hyperscaling inequality $\eta_1 \le d/(1+\delta)$, where $\delta$ is the critical exponent governing the volume of critical clusters, with the mean-field bound $\delta \ge 2$~\cite{ab87}. In sufficiently high dimension $d \ge 11$, it is known that $\eta_1$ takes its mean-field value $\eta_1 = 2$ \cite{kn11,fh17}, see also Corollary \ref{s:ce} below. Note that the bound $\eta_1 \le d/3$ is expected to be tight in dimension $d_c=6$ (indeed, it implies that the upper critical dimension $d_c$ is at least $6$ since it shows that $\eta_1 < 2$ for $d \le 5$).

\item For general self-dual Bernoulli models in dimension $d=2$, the best-known bound is $\eta_1 \le 1/3$ \cite{kes87}\footnote{The argument is attributed to van den Berg.} Sharper results are known for two very specific models, namely $\eta_1 \le 1/6$ for bond percolation on $\Z^2$ \cite{dcmt20}, and the exact value $\eta_1 = 5/48$ for site percolation on the triangular lattice \cite{sw01}.
\end{enumerate} 

Previously for Gaussian percolation it was only known that $\eta_1 \le d-1$ and $\eta_1 \le 1/2$ in the self-dual case; the latter is a consequence of RSW estimates \cite{bg17, rv19, mv20}, and see \cite{dg20} for the~former.
\end{remark}

\begin{remark}
We emphasise that the bound $\eta_1 \le d/3$ in Theorem \ref{t:bou} does 
\textit{not} assume positive correlations, and so applies to a class of models that lack positive associations.
\end{remark}

\begin{remark}
\label{r:limsup}
Unlike in the self-dual case, our proof that $\eta_1 \le d/3$ does \textit{not} allow us to deduce the pointwise lower bound
\[ \mathbb{P}_{\ell_c}[ \Lambda_1 \longleftrightarrow \partial \Lambda_R] \ge c R^{-d/3} , \]
since we establish $\eta_1 \le d/3$ by obtaining a contradiction with $\mathbb{P}_{\ell_c}[ \Lambda_1 \longleftrightarrow \partial \Lambda_R] \le c R^{-d/3-\varepsilon}$. This is also true for the previous argument establishing $\eta_1 \le d/3$ for Bernoulli percolation \cite{bcks99,ab87}.
\end{remark}

\begin{remark}
\label{r:larm}
In the self-dual case we also obtain a bound on the `one-arm exponent' for \textit{level lines}, namely that there exists a $c > 0$ such that, for all $R \ge 1$,
\begin{equation}
    \label{e:level}
\P_{\ell_c}\big[ \{ \text{there exists a path in $\{f = 0\}$ that intersects $\Lambda_1$ and $\partial \Lambda_R$} \}  \big] \ge c R^{-2/3} .
\end{equation}
To deduce $\eta_1 \le 1/3$ we use the observation that, by the FKG inequality and self-duality, the left-hand side of~\eqref{e:level} is bounded above by the square of $\mathbb{P}_{\ell_c}[ \Lambda_1 \longleftrightarrow \partial \Lambda_R]$, see \eqref{e:eta1eta2}.
\end{remark}

\begin{remark}
Clearly if $q$ has bounded support then $f$ is finite-range dependent, but we do not know whether every finite-range $f$ can be represented as $q \star W$ for $q$ with bounded support (although this seems very natural). Certainly this is true assuming positive correlations or if $d=1$  \cite{egr04}. It is also known \cite{rud70} that if $f$ is finite-range and \textit{isotropic} (i.e.\ its law is rotationally symmetric) then it can be represented as a \textit{countable sum} of independent $f_i = q_i \star W_i$ for $q_i$ with \textit{uniformly bounded} support. Since it is straightforward to extend our proof to handle such fields, the conclusion of Theorem \ref{t:bou} (and Theorem~\ref{t:expdecay} below) are also true if we replace the assumption that $q$ has bounded support with the assumption that $f$ is finite-range dependent, and is either positively-correlated or isotropic.
\end{remark}

Our next result extends Theorem \ref{t:bou} to a class of models with unbounded range of dependence. This depends on the decay exponent $\beta > d$ in Assumption \ref{a:gf}, and recovers the bounds in Theorem~\ref{t:bou} in the $\beta \to \infty$ limit.

Recall that the mean-field bound \eqref{a:mflb} is expected to hold in general for Gaussian percolation. In this paper we prove it only for finite-range models, see Theorem~\ref{t:expdecay} below; for general models we instead introduce it as an assumption.

\begin{theorem}[Models with polynomial correlation decay]
\label{t:oagf}
Suppose $f$ satisfies Assumption \ref{a:gf} with parameter $\beta > d$, and~\eqref{a:pos1}.
\begin{enumerate}
\item If $d\ge3$ and $\beta >  4d-4$, then $\eta_1 \leq  \min\{ \frac{d-2 + \alpha(3d-1)}{1+2\alpha}, d-1 \}$ where $\alpha =  \frac{3d-4}{2\beta-5d+4}$.
\item If $d\ge2$, $\beta > 2d-1$, and \eqref{a:mflb} holds, then $\eta_1  \le  \min\{ \max\{ \frac{d}{3} +  \frac{\alpha(d-1)}{3},  \frac{\alpha(2d-1)}{3} \} , d-1\}$ where $\alpha =  \frac{3d-2}{2\beta-d}$.
\item If $f$ is continuous and $d=2$, $\beta > \frac83$, and \eqref{a:pos2} holds, then $\eta_1 \le  \min\{ \frac{1}{3} + \frac{5}{6(\beta-1)}, \frac12 \} $.
\end{enumerate}
\end{theorem}

\begin{corollary}[Models with faster-than-polynomial correlation decay]
\label{c:oagf}
Suppose $f$ satisfies Assumption \ref{a:gf} for every parameter~$\beta$. 
\begin{enumerate}
\item If $d\ge3$ then $\eta_1 \le d-2$.
\item If $d\ge2$ and \eqref{a:mflb} holds, then $\eta_1  \le  \frac{d}{3}$.
\item If $f$ is continuous and $d=2$, and if \eqref{a:pos2} holds, then $\eta_1 \le 1/3$.
\end{enumerate}
\end{corollary}
\begin{proof}
Take $\beta \to \infty$ in Theorem \ref{t:oagf}.
\end{proof}

To illustrate Corollary \ref{c:oagf} we consider two important fields to which it applies:

\begin{example}
The \textit{Bargmann-Fock} field is the continuous stationary Gaussian field on $\R^d$ with covariance kernel $K(x) = e^{-|x|^2/2}$ (see \cite{bg17} for background and motivation). This field satisfies Assumption \ref{a:gf} for every parameter~$\beta$ and also \eqref{a:pos2}, hence Corollary~\ref{c:oagf} applies, although it is not yet known whether \eqref{a:mflb} holds.
 \end{example}
 
 \begin{example}
 The \textit{massive Gaussian free field} (MGFF) is the stationary Gaussian field on $\Z^d$ whose covariance kernel is the Green function of a simple random walk on $\Z^d$ killed at exponential rate $m>0$ (see \cite{rod17} for a precise definition and motivation). We claim that the MGFF satisfies Assumption \ref{a:gf} for every parameter~$\beta$. Indeed the spectral density of the MGFF is proportional to $\rho(s_1,\ldots,s_d) =  (m  + 1 - d^{-1} \sum_i \cos(s_i)) )^{-1}$, and since $\sqrt{\rho}$ is smooth with all derivatives absolutely integrable, $q = \mathcal{F}[\sqrt{\rho}]$ decays faster than any polynomial. Since also \eqref{a:pos2} holds, Corollary~\ref{c:oagf} applies. Although it is not yet known that \eqref{a:mflb} holds, we believe that the approach in \cite{dcrt19a} can be adapted to prove this (see the comments in \cite[Section 1.2]{dgrs20}).
 \end{example}

\subsection{Relations to other critical exponents}
\label{s:ce}

Our methods also give bounds on $\eta_1$ in terms of other critical exponents, some of which appear to be new even for Bernoulli percolation.

For simplicity we state these results only in the case of Bernoulli site percolation, although similar results to Theorems \ref{t:bou} and \ref{t:oagf} could be proven under more general assumptions. The proofs also go through unchanged for Bernoulli \textit{bond} percolation. For the remainder of the subsection we assume that $\X = \Z$ and $q(x) = \id_{x = 0}$.

\smallskip
Let us introduce the relevant exponents. Recall the mean-field bound  \eqref{a:mflb} on $\theta(\ell)$. It is expected that $\theta(\ell) \to 0$ as a power law as $\ell \downarrow \ell_c$; although this is not known rigorously (except in high dimension), we will assume that the corresponding exponent exists
\[ \beta = \lim_{\ell \downarrow \ell_c} \frac{ \log \theta(\ell)}{\log |\ell_c-\ell|} \in (0, \infty) . \]
 Below criticality $\ell < \ell_c$, it is known that connection probabilities decay exponentially \cite{men86,ab87,dct16} and that the limit
\[  \frac{1}{\xi(\ell)} :=   \lim_{R \to \infty} \frac{ - \log \mathbb{P}_\ell[  \Lambda_1 \longleftrightarrow\partial \Lambda_R    ] }{R} \in (0, \infty)      \]
exists \cite[Theorem 6.10]{gri99}. The \textit{correlation length} $\xi(\ell)$ is expected to diverge as a power law as $\ell \uparrow \ell_c$, and we will again assume that the corresponding exponent exists
\[   \nu =  \lim_{\ell \uparrow  \ell_c}  \frac{ - \log \xi(\ell) }{\log |\ell_c-\ell|} \in (0, \infty)  . \]
Similarly, as $\ell \uparrow \ell_c$ the \textit{susceptibility} $\chi(\ell)  :=  \sum_{v \in \Z^d} \P_\ell[ 0 \longleftrightarrow v]  < \infty$ is expected to diverge as a power law, and we will assume the existence of
\[   \gamma =  \lim_{\ell \uparrow \ell_c}  \frac{ - \log \chi(\ell) }{\log |\ell_c-\ell|}   \in (0, \infty) . \]
Finally we also assume that the critical \textit{two-point function} decays as a power law with exponent
\[  d - 2 + \eta := \lim_{|v|_\infty \to \infty} \frac{ -\log \mathbb{P}_{\ell_c}[ 0 \longleftrightarrow v] }{\log |v|_\infty}  \in (0, \infty),\]
where $|\cdot|_\infty$ denotes the sup-norm. We remark that $\beta \le 1$ because of \eqref{a:mflb}, and it is well known that $\nu \ge 2/d$~\cite{ccfs86}, $\gamma \ge 1$ \cite{an84}, and $\eta \le 1$ \cite{ham57}.

\begin{theorem}
\label{t:oab2}
For Bernoulli percolation on $\Z^d$, assuming the existence of $\beta, \nu, \gamma$ and $\eta$,
\begin{equation}
\label{e:new}
  \frac{2-\gamma}{\nu} \le \eta_1 \le \bar{\eta}_1 \le  \min\Big\{ d - \frac{2}{\nu}, \frac{2-\eta}{2/\beta-1}  \Big\} \le 2-\eta  ,
  \end{equation}
where $\bar{\eta}_1$ is defined as in \eqref{e:eta} with limsup replacing liminf. Moreover $\eta_1 \le   \frac{d}{2/\beta+1}$.
\end{theorem}

\begin{remark}
 To our knowledge the bounds in \eqref{e:new} are new even for Bernoulli percolation, and $\eta_1 \ge \frac{2-\gamma}{\nu}$ may be of particular interest as a \textit{lower} bound on $\eta_1$. The bound $\eta_1 \le   \frac{d}{2/\beta+1}$ is implied by the hyperscaling inequality in \cite{bcks99}.
\end{remark}

In sufficiently high dimension it is known that the exponents $\nu$, $\gamma$ and $\eta$ exist and take their mean-field values $\nu = 1/2$ \cite{har90}, $\gamma = 1$ \cite{an84}, and $\eta=0$ \cite{har08}. Hence Theorem \ref{t:oab2} gives a new proof of the result of Kozma and Nachmias that $\eta_1 = 2$ in high dimension:

\begin{corollary}[\cite{kn11}]
\label{c:oab}
For Bernoulli percolation on $\Z^d$, there exists $d_0 > 0$ such that, if $d \ge d_0$, 
\[   \lim_{R \to \infty} \frac{- \log \P_{\ell_c}  [ \Lambda_1 \longleftrightarrow \partial \Lambda_R ]  }{\log R}  = 2 . \]
\end{corollary}

\begin{remark}
\label{r:high} 
Our argument is significantly simpler than the one in \cite{kn11}, however it yields only
\[   c_1 R^{-2} \le  \P_{\ell_c} [ \Lambda_1 \longleftrightarrow \partial \Lambda_R ] \le  c_2 R^{-2}  (\log R)^4  \]
whereas \cite{kn11} proved that $\P_{p_c} [ \Lambda_1 \longleftrightarrow \partial \Lambda_R ]  \asymp R^{-2}$ in the sense of bounded ratios (see Remark~\ref{r:high2}). Another difference is that we deduce $\eta_1 = 2$ in any dimension from the bounds $\nu \le 1/2$ and $\eta \ge 0$ (see Remark \ref{r:high2}, or recall the Fischer inequality $\gamma/\nu \le 2-\eta$), whereas the argument in~\cite{kn11} uses as input $d > 6$ and the two-sided bound $\eta = 0$ (or more precisely $ \mathbb{P}_{\ell_c}[ 0 \longleftrightarrow v] \asymp |v|_\infty^{-d+2}$).
\end{remark}

\subsection{Sharpness of the phase transition for finite-range models}
\label{ss:expdecay}

As part of the proof of Theorem~\ref{t:bou} we establish the mean-field bound \eqref{a:mflb} for finite-range models. For this we adapt the approach of Duminil-Copin, Raoufi and Tassion \cite{dcrt19a}; as in~\cite{dcrt19a} (and related contexts \cite{ab87,men86}), this argument also yields the \textit{sharpness} of the phase transition.

\begin{theorem}[Sharpness of the phase transition for finite-range models]
\label{t:expdecay}
Suppose $f$ satisfies Assumption \ref{a:gf} and $q$ has bounded support. Then the mean-field bound \eqref{a:mflb} holds. Moreover, for every $\ell < \ell_c$ there exist  $c_1, c_2 > 0$ such that, for $R \ge 1$,
\begin{equation}
    \label{e:expdecay}
 \mathbb{P}_\ell[  \Lambda_1 \longleftrightarrow \partial \Lambda_R  ]  \le c_1 e^{-c_2 R} . 
 \end{equation}
\end{theorem}
 
\begin{remark}
\label{r:known3}
For Gaussian percolation, the exponential decay in \eqref{e:expdecay} was only known so far in two cases: (i) in the self-dual case assuming positive correlations \cite{mv20}, and (ii) for certain positively-correlated Gaussian fields on $\mathbb{Z}^d$ \cite{dgrs20,grs22}. The mean-field bound \eqref{a:mflb} was not yet known for any Gaussian models.

We emphasise that in Theorem \ref{t:expdecay} we do not assume positive correlations, so this theorem holds for a class of models lacking positive associations. The proof also does not require the symmetry condition in Assumption \ref{a:gf}.
\end{remark}

\subsection{Elements of the proof: the `entropic' bound, and a new Russo-type inequality}

The existing upper bounds on $\eta_1$ for Bernoulli percolation rely heavily on specific properties of this model (such as the BK inequality, or for the sharper results on specific lattices, also on properties of parafermionic observables) and do not extend easily to dependent percolation models. We develop a new approach based on exploration and relative entropy arguments.

\smallskip
There are two main ingredients in our approach, giving upper and lower bounds respectively on the change in probability of events as $\ell$ varies:
\begin{enumerate}
    \item A non-differential `entropic' upper bound, with control in terms of revealment probabilities of explorations.
    \item A `Russo-type inequality' giving a lower bound on the derivative for increasing events, which when combined with the OSSS inequality also gives control in terms of revealment probabilities of explorations.
\end{enumerate}
In the self-dual case we compare these inequalities when applied to square-crossing events at criticality; by a judicious choice of algorithm the revealment probabilities are controlled by one-arm probabilities for level lines (c.f.\ Remark \ref{r:larm}). In the general case we first use the Russo-type inequality to establish the mean-field bound \eqref{a:mflb}, and then compare this to the entropic bound applied to arm events in the critical and slightly supercritical regimes.

\smallskip
Since we believe these ingredients to be of independent interest, we describe them in detail:

\subsubsection{The entropic bound}

For context, let us first state a well-known bound for Bernoulli percolation. Recall the concept of a \textit{randomised algorithm} (or `decision tree' or `exploration' depending on the context):

\begin{definition}[Randomised algorithms]
\label{d:ra}
Let $X =  (X_i)_{i \in \N}$ be a countable set of random variables taking values in arbitrary measurable spaces, and let $Z = (Z_i)_{i \in \N}$ be an auxiliary sequence of independent uniform random variables. A \textit{(randomised) algorithm} $\mathcal{A}$ on $X$ is a random sequence $X^\mathcal{A} = (X_{\varphi(1)},..,X_{\varphi(i)},..)$ such that, for each $i\in\N$, $\varphi(i) \in \N$ is distinct from all $\varphi(j)$, $j < i$, and depends only on (i) $(X_{\varphi(1)},..,X_{\varphi(i-1)})$, and (ii) the auxiliary randomness $Z_i$. The sequence $X^\mathcal{A}$ can be either infinite or finite; in the latter case we say that $\mathcal{A}$ \textit{terminates}, and a real number is returned which is a measurable function of $X^\mathcal{A}$. We say that $\mathcal{A}$ \textit{determines an event} $A$ if it almost surely terminates and returns the value $\id_A$. The \textit{revealment probability} $\textrm{Rev}(i)$ is the probability that  $\mathcal{A}$ \textit{reveals} $X_i$, i.e.\ that $X^\mathcal{A}$ contains the coordinate $X_i$.
\end{definition}

The following bound is essentially contained in \cite{os07} (see also \cite[Appendix B]{ss11} and \cite{vn20} for similar bounds):

\begin{proposition}
\label{p:ebber}
Let $Y = (Y_i)_{i \le n}$ be a finite set of independent Bernoulli random variables, and let $\P_p[\cdot]$ denote its law under parameter $p \in [0, 1]$. Let $A$ be an event, and let $\mathcal{A}$ be an algorithm on $Y$ determining $A$. Then for all $p \in (0,1)$,
\[ \Big|  \frac{d}{d p}  \mathbb{P}_p[A] \Big| \le \frac{1}{p(1-p)}  \sqrt{ \P_p[A] \E_p [N] } , \]
where $N$ is the number of coordinates that are revealed by $\mathcal{A}$.
\end{proposition}

\noindent We extend Proposition \ref{p:ebber} in two ways:
\begin{itemize}
    \item We give a non-differential version bounding $|\P_q[A]-\P_p[A]|$, $q > p$, in terms of the expected number of revealed coordinates $\E_p[N]$ when the algorithm \textit{is run at parameter $p$}. This is a sharper bound whenever the algorithm is more efficient when run at parameter $p$ compared to parameter $q$.
    \item We give a new proof based on properties of the relative entropy, which applies in a variety of settings. See Remark~\ref{r:alt} for a comparison with the approach in \cite{os07,ss11,vn20}.
\end{itemize}

To state this extension we return to the Gaussian setting. Fix a constant $s > 0$ and partition $\mathbb{R}^d$ into boxes $\mathcal{S}_s$ which are translations of $[0, s)^d$ by the lattice $s \mathbb{Z}^d$. For $S \in \mathcal{S}_s$, let $W|_S = W \id_S$ denote the restriction of the white noise $W$ to $S$, and consider the decomposition
\begin{equation}
    \label{e:orthdecomp}
  f = \sum_{S \in \mathcal{S}_s} f_S \ , \quad f_S =  q \star W|_S  ,
\end{equation}
where each $f_S$ is an independent centred Gaussian field (in the case $\X = \R^d$, we also assume that each $f_S$ is simultaneously continuous almost surely, which is possible by countability). Let $\mathcal{A}_s$ denote the collection of randomised algorithms on $(f_{S})_{S \in \mathcal{S}_s}$. We say that a box $S \in \mathcal{S}_s$ is \textit{revealed} by an algorithm $\mathcal{A} \in \mathcal{A}_s$ if $f_S$ is revealed, with $\text{Rev}(S)$ the corresponding probability.

\begin{proposition}[Entropic bound]
\label{p:entbound}
Suppose $f = q \star W$ is either a Gaussian field on $\Z^d$ or a continuous Gaussian field on $\R^d$, and assume that $q$ has bounded support and $\int q > 0$. Then for every $\ell \in \R$, compactly supported event $A$, $s > 0$, algorithm $\mathcal{A} \in \mathcal{A}_s$ that determines $A$, and $\varepsilon \ge 0$,
\[  \big| \P_{\ell + \varepsilon} \big[A \big] - \P_\ell[ A]  \big| \le \frac{\varepsilon s^{d/2} }{\int q} \! \!  \sqrt{  \max \{ \P_\ell[A], \P_{\ell + \varepsilon}  [ A  ]  \} \E_\ell  [N] },  \]
where $N$ is the number of boxes in $\mathcal{S}_s$ that are revealed by $\mathcal{A}$ .
\end{proposition}

Notice that Proposition \ref{p:entbound} compares the field $f + \ell$ and its perturbation by $\varepsilon > 0$. In fact we use an extension of this result (see Proposition~\ref{p:entbound2}) which allows for more general perturbations. For the reader interested primarily in Bernoulli percolation, we record the extended form of the entropic bound in that setting here:

\begin{proposition}[Entropic bound for Bernoulli percolation]
\label{p:ebber2}
Let $Y$, $A$ and $\mathcal{A}$ be as in Proposition \ref{p:ebber}. Let $I \subset \{1, \ldots, n\}$ be a subset of coordinates, let $p,q \in (0,1)$, and let $\P_{p; q; I}[\cdot]$ denote the law of $Y$ where the Bernoulli parameter of coordinate $Y_i$ is $q$ if $i \in I$ and $p$ if $i \notin I$. Then
\begin{equation*}
|\mathbb{P}_{p;q;I}[A]  - \mathbb{P}_p[A] | \le  \max \! \Big\{ \!   \frac{1}{\sqrt{p(1-p)}}, \frac{1}{\sqrt{q(1-q)}} \! \Big\} |p-q|  \sqrt{  \max\{ \mathbb{P}_p[A] , \mathbb{P}_{p;q;I}[A] \}   \E_p [N_I]    } ,
\end{equation*}
where $N_I$ is the number of coordinates in $I$ that are revealed by $\mathcal{A}$.
\end{proposition}

\noindent Setting $I = \{1, \ldots , n\}$ and taking $q \downarrow p$ recovers Proposition \ref{p:ebber} (with improved constant $1/\sqrt{p(1-p)}$ in place of $1/(p(1-p))$).

\subsubsection{The Russo-type inequality}
Recall the classical \textit{Russo formula} for Bernoulli percolation: for $p \in (0,1)$ and an increasing event $A$,
\begin{equation}
\label{e:russoresamp}
 \frac{d}{d p}  \mathbb{P}_p[A]    = \frac{2}{p(1-p)} \sum_i \Infl_{p;A}(i)  ,
\end{equation}
where $Y = (Y_i)_{i \le n}$ and $\P_p[\cdot]$ are as in Proposition \ref{p:ebber}, and
\[   \Infl_{p;A}(i) :=   \mathbb{P}_p \big[ \id_{Y \in A} \neq \id_{Y^{(i)} \in A} \big]  \]
denotes the \textit{resampling influence}, where $Y^{(i)}$ denotes a copy of $Y$ with the coordinate $Y_i$ resampled independently (for increasing events, the resampling influence coincides with the classical `pivotal probability' up to a constant). By combining \eqref{e:russoresamp} with the OSSS inequality $\Var_p [ \id_A   ]  \le  \frac{1}{2} \sum_{i} \textrm{Rev}(i) \Infl_{p;A}(i)$ \cite{osss05}, one obtains the bound
   \[  \frac{d}{d p}  \mathbb{P}_p[A] \ge  \frac{4  }{p(1-p)} \frac{\Var_p [ \id_A  ]}{  \max_i \textrm{Rev}(i) } ,  \]
where the revealment probabilties are under $\P_p[\cdot]$.

In order to apply similar arguments in Gaussian percolation, one wishes to have a Gaussian extension of \eqref{e:russoresamp}; in fact one only needs a \textit{lower bound} in place of the equality.

\smallskip 
In \cite{mv20} such an extension was given for discrete Gaussian fields with finite-range dependence (i.e.\ the case $\X = \Z$ and $q$ with bounded support). As in the previous subsection, fix $s > 0$ and recall the orthogonal decomposition $f = \sum_{S \in \mathcal{S}_s} f_S$ in \eqref{e:orthdecomp}. For an event $A$ and $S \in \mathcal{S}_s$, we define the \textit{resampling influence}
\begin{equation}
    \label{e:infl}
 \Infl_{\ell;A}(S) :=   \mathbb{P}_\ell \big[ \id_{\{f \in A\}} \neq \id_{\{f^{(S)} \in A\}} \big]  ,
 \end{equation} 
where $f^{(S)}$ denotes the field $f = \sum_{S \in \mathcal{S}_s} f_S$ with $f_S$ resampled independently. The result in \cite{mv20} concerned the case $\X = \Z$ with $s = 1$, so that resampling $f_S$ is equivalent to resampling the one-dimensional Gaussian $W|_S$. We state a version of this result:
 
 \begin{proposition}[Russo-type inequality, discrete $L^1$-version; {\cite[Proposition 5.3]{mv20}}]
 \label{p:russodis}
 Suppose $f = q \star W$ is a Gaussian field on $\Z^d$ and $q \ge 0$ has bounded support. Then there exists a constant $c > 0$ depending only on $d$ such that, for every $\ell \in \R$ and increasing compactly supported event~$A$,
\[ \frac{d}{d \ell} \mathbb{P}_\ell[A]  \ge \frac{c }{\|q\|_1} \sum_{S \in \mathcal{S}_1}  \Infl_{\ell;A}(S)  .\]
 \end{proposition}
\noindent Observe that setting $q(x) = \id_{x =0}$ in Proposition \ref{p:russodis} recovers \eqref{e:russoresamp} as a lower bound, up to a constant.
 
 \smallskip
 Although this result concerns discrete Gaussian fields, in \cite{mv20} it was applied to continuous fields by first \textit{discretising} the white noise $W$ at scale $\varepsilon \ll 1$. However the presence of the `$L^1$-type' factor $\|q\|_1$ makes the bound degenerate in the $\varepsilon \to 0$ limit. Although this issue was circumvented in the specific application in \cite{mv20}, it prevented an adaptation of \cite{dcrt19a} to prove Theorem \ref{t:expdecay} for continuous fields, even finite-range dependent.
 
 \smallskip
 In this paper we establish an improvement on Proposition \ref{p:russodis} which does not suffer from this degeneracy, and applies directly to discrete and continuous fields. We also remove the requirement that $q \ge 0$. To achieve this we analyse the effect of resampling \textit{large regions of white noise simultaneously}. Interpreting the Russo-type inequality as a consequence of the Gaussian isoperimetric inequality, we then exploit the fact that Gaussian isoperimetry is dimension-free.
 
 \smallskip
 We say that an event $A$ is a \textit{continuity event} if $\ell \mapsto \mathbb{P}_\ell[ f + g|_{\X^d} \in A ]$ is continuous for every smooth function $g : \R^d \to \R$. Recall the \textit{Dini derivatives}, defined for $g: \R \to \R$ as
\begin{equation}
    \label{e:dini}
    \frac{d^+}{dx}g(x)=\limsup\limits_{\varepsilon \downarrow 0}\frac{g(x+\varepsilon)-g(x)}{\varepsilon} \quad \text{and} \quad \frac{d^-}{dx}g(x)=\liminf\limits_{\varepsilon \downarrow 0}\frac{g(x+\varepsilon)-g(x)}{\varepsilon} .
    \end{equation}

\begin{proposition}[Russo-type inequality, $L^2$-version]
\label{p:russo}
Suppose $f = q \star W$ is either a Gaussian field on $\Z^d$ or a continuous Gaussian field on $\R^d$, and let $r  >  0$ be such that $q$ is supported on $\Lambda_r$. Then there exists a constant $c > 0$ depending only on $d$ such that, for every $\ell \in \R$, $s > 0$, and increasing compactly supported continuity event $A$,
\[ \frac{d^-}{d \ell} \mathbb{P}_\ell[A]  \ge \frac{c \min\{1, (s/r)^d\}}{\|q\|_2} \sum_{S \in \mathcal{S}_s}  \Infl_{\ell;A}(S) .\]
\end{proposition}

\begin{remark}
In the literature there have been other approaches to obtaining Russo-type inequalities for Gaussian percolation using alternative measures of influence, see e.g.\ \cite[Proposition 3.2]{rod17} and \cite[Theorem 2.19]{rv20}. However these also suffer from degeneracy issues when passing to the continuum, and do not combine easily with the OSSS inequality.
\end{remark}

\subsubsection{A general lower bound for revealment probabilities}
\label{s:genbound}

Recall that in the self-dual case we obtain bounds on $\eta_1$ by comparing the entropic bound Proposition \ref{p:ebber} with the Russo-type inequality in Proposition \ref{p:russo} (and the OSSS inequality) for a well-chosen algorithm.

\smallskip
A similar approach appeared previously in the general setting of increasing Boolean functions~\cite{bsw05}, where it was used to obtain the following result:

\begin{proposition}[{\cite[Theorem 2 (part 2)]{bsw05}}]
\label{p:genrev}
Let $Y = (Y_i)_{i \le n}$, $A$ and $\mathcal{A}$ be as in Proposition~\ref{p:ebber}, and suppose that $A$ is increasing. Then
\[ \max_{i \le n} \textrm{Rev}(i) \ge \frac{ \Var_{1/2}[ \id_A ]^{2/3}}{ n^{1/3}}   , \]
where the revealment probabilties are under $\P_p[\cdot]$.
\end{proposition}

As far as we know one cannot recover the bound $\eta_1 \le 1/3$ in the self-dual case, even for Bernoulli percolation, simply by applying the result in Proposition \ref{p:genrev} to a well-chosen algorithm. Instead one needs a `conditional' version that bounds the maximum revealment probabilities within a subset of coordinates. Since we believe it to be of independent interest, we state this bound in the Bernoulli setting. See Remark \ref{r:genrev} for comments on the Gaussian version.

\begin{proposition}
\label{p:genrev2}
Let $Y$, $A$ and $\mathcal{A}$ be as in Proposition \ref{p:genrev}, let $I \subset \{1, \ldots, n\}$ be a subset of coordinates, and let $\mathcal{F}_I$ be the $\sigma$-algebra generated by $(Y_i)_{i \in I}$. Then
\[   \max_{i \in I} \textrm{Rev}(i) \ge   \frac{ \big(4  \Var_p \big[ \, \P_p[A  \: | \: \mathcal{F}_I ]  \big] \big)^{2/3} }{ (p(1-p) \P_p[A] |I|)^{1/3}}   ,\]
where the revealment probabilties are under $\P_p[\cdot]$.
\end{proposition}
\noindent Setting $p = 1/2$ and $I = \{1,2,\ldots ,n\}$ recovers Proposition \ref{p:genrev2} with an improved constant.

\begin{proof}
We use the conditional version of the OSSS inequality (see the proof of Proposition~\ref{p:lbderiv} for how to derive this from the standard version)
\[   \Var_p \big[ \P_p[A  \: | \: \mathcal{F}_I ]  \big]  \le  \frac{1}{2} \sum_{i \in I} \textrm{Rev}(i) \Infl_{p;A}(i)  . \]
  Combining with Russo's formula \eqref{e:russoresamp} (applied to the subset $I$) gives
\begin{equation}
    \label{e:genrev1}
\sum_{i \in I} \frac{\partial}{\partial p_i} \P_p[A]  = \frac{2}{p(1-p)} \sum_{i \in I}  \Infl_{p;A}(i)  \ge \frac{4}{p(1-p)} \frac{  \Var_p \big[ \, \P_p[A  \: | \: \mathcal{F}_I ]  \big] }{  \max_{i \in I} \textrm{Rev}(i) }  ,
\end{equation}  
where $\partial/\partial p_i$ denotes the derivative with respect to changing the Bernoulli parameter at coordinate $i$. On the other hand, Proposition \ref{p:ebber} (with $q \to p$) gives 
\begin{equation}
    \label{e:genrev2}
  \sum_{i \in I} \frac{\partial}{\partial p_i} \P_p[A] \le \frac{1}{\sqrt{p(1-p)}} \sqrt{\P_p[A] \E_p |N_I| }  \le \frac{1}{\sqrt{p(1-p)}} \sqrt{\P_p[A] |I| \max_{i \in I} \textrm{Rev}(i) } . \end{equation}
 Combining \eqref{e:genrev1}--\eqref{e:genrev2} and rearranging proves the result.
\end{proof}

 \begin{remark}
If $p=1/2$, the quantity $\Var_p [  \P_p[A   |  \mathcal{F}_I ]  ] $ has an interpretation as the square-sum of the Fourier coefficients of $\id_A$ supported on non-empty subsets of $I$ (see, e.g., \cite{gps10}).
\end{remark}

\subsubsection{Other models}
Other than Bernoulli and Gaussian percolation, our approach adapts naturally to many other models in the Bernoulli percolation universality class. For instance, both Poisson-Voronoi and Poisson-Boolean percolation could be treated in a similar way (although in the latter case the bounds on $\eta_1$ may depend on the decay of the radius distribution, and note that only the former model is self-dual in dimension $d=2$). For brevity we do not discuss details here, although see \cite{dm21b} for a recent application of the entropic bound to the Poisson-Boolean model.

\subsection{Outline of the paper} 
In Section \ref{s:ent} we establish the entropic bound in Proposition \ref{p:entbound}. In Section \ref{s:bou} we study finite-range models and give the proof of Theorems \ref{t:bou} and \ref{t:oab2}. In Section \ref{s:gf} we adapt the arguments to models with unbounded range of dependence, and give the proof of Theorem \ref{t:oagf}. In Section \ref{s:osss} we establish the Russo-type inequality in Proposition \ref{p:russo}, and apply it to prove Theorem \ref{t:expdecay}. The appendix contains a technical result on orthogonal decompositions of Gaussian fields.

\subsection{Acknowledgements}
The second author was partially supported by the Australian Research Council (ARC) Discovery Early Career Researcher Award DE200101467. The authors thank Damien Gayet, Tom Hutchcroft, Ioan Manolescu and Hugo Vanneuville for helpful discussions, comments on an earlier draft, and for pointing out references \cite{dcmt20} (Ioan) and \cite{bsw05,bd20} (Hugo), and an anonymous referee for valuable comments that helped improve the presentation of the paper. This work was initiated while the first author was visiting the second author at Queen Mary University of London and we thank the University for its hospitality.

%%%%%%%%
\medskip
\section{The entropic bound}
\label{s:ent}

In this section we prove the entropic bound in Proposition \ref{p:entbound}, or rather a slight extension of this bound. Suppose $f = q \star W$ is either a Gaussian field on $\Z^d$ or a continuous Gaussian field on $\R^d$, and assume that $q$ has bounded support.

\begin{proposition}[Entropic bound; extended version]
\label{p:entbound2}
For every $\ell \in \R$, compactly supported event $A$, $s > 0$, algorithm $\mathcal{A} \in \mathcal{A}_s$ that determines $A$, subset $\mathcal{S}' \subseteq \mathcal{S}_s$, and $\varepsilon \ge 0$,
\[ \Big| \P_\ell \Big[f+ \varepsilon \! \sum_{S \in \mathcal{S}'} q \star \id_S \in A \Big] - \P_\ell[f\in A]  \Big| \le \varepsilon s^{d/2} \! \!  \sqrt{  \max\Big\{ \P_\ell[A], \P_\ell \big[f+ \varepsilon \! \sum_{S \in \mathcal{S}'} q \star \id_S \in A \big] \Big\} \E_\ell  [N_{\mathcal{S}'}] }, \]
where $N_{\mathcal{S}'}$ is the number of boxes in $\mathcal{S}'$ that are revealed by $\mathcal{A}$.
\end{proposition}

\begin{proof}[Proof of Proposition \ref{p:entbound} assuming Proposition \ref{p:entbound2}]
Since  $\! \sum_{S \in \mathcal{S}_s} q \star \id_S = \int q$, setting $\mathcal{S}' = \mathcal{S}_s$ in Proposition \ref{p:entbound2} and replacing $\varepsilon \mapsto \varepsilon / \int q$ gives Proposition \ref{p:entbound}.
\end{proof}

\subsection{The relative entropy}
Our proof of Proposition \ref{p:entbound2} relies on some properties of the relative entropy which we recall here. For $P$ and $Q$ probability measures on a common measurable space, the \textit{relative entropy (or Kullback-Leibler divergence) from $P$ to $Q$} is defined as
\[   D_{KL}(P || Q) := \int \log \Big( \frac{dP}{dQ} \Big) \, dP  \]
if $P$ is absolutely continuous with respect to~$Q$, and $D_{KL}(P || Q) := \infty$ otherwise; $D_{KL}(P || Q)$ is non-negative by Jensen's inequality. If $X$ and $Y$ are random variables taking values in a common measurable space, with respective laws $P$ and $Q$, we also write $D_{KL}(X || Y)$ for $D_{KL}(P||Q)$. We shall need two basic properties of the relative entropy (see \cite[Theorem 2.2 and Corollary 3.2]{kul78}):
\begin{enumerate}
\item(Chain rule) Let $X = (X_1, X_2)$ and $Y = (Y_1, Y_2)$ be $\R^{k_1 + k_2}$-valued mutually absolutely continuous random variables which have densities with respect to the Lebesgue measure. Then 
\begin{equation}
\label{e:chain}
D_{KL}(X || Y) = D_{KL}(X_1 || Y_1) +  \mathbb{E}_{x \sim X_1} \big[  D_{KL}( (X_2 | X_1 = x)  || (Y_2 | Y_1 = x ) ) \big] . 
\end{equation}
This expression is well-defined since $G(x) := D_{KL}( (X_2 | X_1 = x)  || (Y_2 | Y_1 = x ) )$ is well-defined for an $X_1$-almost sure set of values of $x$.
\item(Contraction) Let $X$ and $Y$ be random variables taking values in a common measurable space and let $F$ be a measurable map from that space. Then 
\begin{equation}
\label{e:contract}
 D_{KL}( X || Y ) \ge  D_{KL}( F(X) || F(Y) )  .
 \end{equation}
\end{enumerate}

\smallskip
We shall also need a simple formula for the relative entropy of stopped sequences of i.i.d.\ random variables. A \textit{stopping time} for a real-valued sequence $X = (X_i)_{i \ge 1}$ is a positive integer $\tau = \tau(X)$ such that $\{\tau \ge n + 1\}$ is determined by $(X_i)_{i \le n}$. We define the corresponding \textit{stopped sequence} $X^\tau = (X^\tau_i)_{i \ge 1}$ as $X^\tau_i = X_i$ for $i \le \tau$, and $X^\tau_i = \dagger$ for $i > \tau$, where $\dagger$ is an arbitrary symbol.

\begin{lemma}
\label{l:kl}
Let $X = (X_i)^n_{i \ge 1}$ and $Y = (Y_i)^n_{i \ge 1}$ be finite real-valued sequences of i.i.d.\ random variables whose respective univariate laws $\mu$ and $\nu$ are mutually absolutely continuous and have densities with respect to the Lebesgue measure. Let $\tau \le n$ be a stopping time, and let $X^\tau$ and $Y^\tau$ be the corresponding stopped sequences. Then
\[ D_{KL} \big(   X^\tau \big\|   Y^\tau \big)  =  \E[\tau(X)]  \, D_{KL}( \mu \| \nu )   .   \] 
\end{lemma}
\begin{proof}
Define $X^{k \wedge \tau} =  (X^\tau_i)_{i \le k}$ and analogously for $Y$. By the chain rule \eqref{e:chain}, for $1 \le k \le n-1$,
\begin{align*}
& D_{KL} \big(   X^{(k+1) \wedge \tau} \big\|   Y^{(k+1) \wedge \tau} \big) \\
&  \quad  =  D_{KL} \big(   X^{k \wedge \tau}  \big\| Y^{k \wedge \tau}  \big)   \! +  \E_{x \sim (X^\tau_i)_{i \le k}} \!  \big[ D_{KL} \! \left( X^\tau_{k+1} \big| (X^\tau_i)_{i \le k} \! = \! x  \big\| Y^\tau_{k+1}  \big| (Y^\tau_i)_{i \le k}  \! = \! x \right)  \! \big]  \\
&  \quad  =  D_{KL} \big(   X^{k \wedge \tau}  \big\| Y^{k \wedge \tau}  \big)   \! +  \E_{x \sim (X^\tau_i)_{i \le k}} \!  \big[  \id_{\tau(X) \ge k+1} D_{KL} \! \left( X_{k+1} \big| (X^\tau_i)_{i \le k} \! = \! x  \big\| Y_{k+1}  \big| (Y^\tau_i)_{i \le k}  \! = \! x \right)  \! \big]  \\
&   \quad =  D_{KL} \big(   X^{k \wedge \tau}  \big\| Y^{k \wedge \tau}  \big)   \! +\P[\tau(X)\ge k+1]  D_{KL}( \mu \| \nu ) \end{align*}
where in the last step we used that $\tau$ is a stopping time. Hence, by induction,
\begin{equation*}
 D_{KL} \big(   X^\tau \big\|   Y^\tau \big)   =   \sum_{1 \le k \le n-1} \P[\tau(X) \ge k+1] D_{KL}( \mu \| \nu )  =    \E[\tau(X)] D_{KL}( \mu \| \nu ) .   \qedhere 
\end{equation*}
\end{proof}

Finally we need a variant of Pinsker's inequality:
\begin{lemma}
\label{l:pinsker}
Let $P$ and $Q$ be probability measures on a common measurable space and let $A$ be an event. Then
\[ |P(A) - Q(A)| \le  \sqrt{ 2 \max\{ P(A), Q(A) \} D_{KL}(P \| Q)} .  \]
\end{lemma}
\begin{proof}
We use a standard reduction to the binary case. Let $\textrm{Ber}(p)$ and $\textrm{Ber}(q)$ be Bernoulli random variables with respective parameters $p := P(A) $ and $q := Q(A)$. By the contraction property \eqref{e:contract} $D_{KL}(P \| Q) \ge D_{KL}( \textrm{Ber}(p) \| \textrm{Ber}(y))$, so it suffices to prove that
\begin{equation}
    \label{e:pinsker}
  (p-q)^2 \le    2 \max\{p,q \}   D_{KL}( \textrm{Ber}(p) \| \textrm{Ber}(q)) . 
  \end{equation} 
If $p \in \{0,1\}$ or $q \in \{0, 1\}$ then \eqref{e:pinsker} is trivial since either the right-hand side is at least $1$ (if $p \neq q$) or both sides are zero (if $p=q$). On the other hand, if $p,q \in (0, 1)$ then
\begin{align}
\label{e:klber}
D_{KL}( \textrm{Ber}(p) \| \textrm{Ber}(q)) & :=   p \log \frac{p}{q} +  (1-p)  \log \frac{1-p}{1-q}   = \int_q^p  \frac{p-s}{s(1-s)} ds   \\
\nonumber &   \ge \frac{1}{  \max\{ p, q \} }  \int_q^p  (p-s)  ds  =   \frac{1}{   2 \max\{ p, q \} }  (p-q)^2  
\end{align}
where we used that $\sup_{ s \in [a, b] } s(1-s) \le \max\{  a, b \} $ for $0 \le a \le b \le 1$.
\end{proof}
\begin{remark}
In the proof we could replace $\max\{ p,q \} $ with $\min\{ \max\{ p,q \} , 1/4 \}$, which recovers the classical Pinsker's inequality $d_{TV}(P, Q) := \sup_A |P(A) - Q(A)| \le  \sqrt{ D_{KL}(P \| Q) / 2 } $.
\end{remark}

\subsection{Proof of Proposition \ref{p:entbound2}}

Consider $S \in \mathcal{S}_s$. We use the decomposition (see Proposition~\ref{p:odecom} in the appendix)
\[ f_S(\cdot) \stackrel{d}{=} \frac{Z_S (q \star \id_S)(\cdot)}{s^{d/2}} + g_S(\cdot), \]
where $Z_S$ is a standard Gaussian random variable and $g_S$ is an independent Gaussian field. This implies also that
\[  f_S(\cdot) + \ell + \varepsilon (q \star \id_S)(\cdot) \stackrel{d}{=} \frac{(Z_S +  \varepsilon s^{d/2})(q \star \id_S)(\cdot)}{s^{d/2}} + \ell + g_S(\cdot) ,\]
in other words, the change in the law of  $f_S(\cdot) + \ell + \varepsilon (q \star \id_S)(\cdot) $ compared to $  f_S(\cdot) + \ell  $ can be induced by adding $\varepsilon s^{d/2}$ to the Gaussian variable $Z_S$.

Let $\mathcal{W}_{\mathcal{S}'}$ denote the subset of $\mathcal{S}_s$ that is revealed by the algorithm, and let $W = (Z_S)_{S \in  \Wc_{\mathcal{S}' }}$ be arranged in order of revealment. Moreover let $W'$ denote the union of $(Z_S)_{S \notin \mathcal{S}'}$ and $(g_S)_{S \in \mathcal{S}_s}$. Finally let $\QQ_\ell$ denote the law of $f + \ell +  \varepsilon \! \sum_{S \in \mathcal{S}'} q \star \id_S$.

First suppose that the algorithm $\mathcal{A}$ does not depend on any auxiliary randomness. Then the event $A$ is measurable with respect to $(W, W')$, and so by Lemma~\ref{l:pinsker}
\begin{equation}
    \label{e:entbound3}
  |\QQ_\ell[A]  - \P_\ell[A] |  \le \sqrt{ 2 \max\{  \P_\ell[A] ,  \QQ_\ell[A]   \}      D_{KL}( (X, V)  || (Y, V) )    ]  }   \end{equation} 
where $(X,V)$ (resp.\ $(Y,V)$) is a random variable with the law of $(W,W')$ under $\P_\ell$ (resp.\ $\QQ_\ell$). Moreover, conditionally on $W'$, $W$ has the law, under $\P_\ell$ (resp.\ $\QQ_\ell$), of a sequence of i.i.d.\ standard Gaussian random variables with mean $0$ (resp.\ $\varepsilon s^{d/2}$) stopped at the stopping time $N_{\mathcal{S}'} = |\mathcal{W}_{\mathcal{S}'}|$. Hence by the chain rule for the Kullback-Liebler divergence and Lemma \ref{l:kl}, 
\begin{equation}
    \label{e:entbound4}
    D_{KL}( (X, V)  || (Y, V) )  =  \E \big[  D_{KL}( (X | \mathcal{F}_V)  || (Y | \mathcal{F}_V) ) \big] =   \E_\ell [N_{\mathcal{S}'}]  D_{KL}(Z  \| Z + \varepsilon s^{d/2} ) 
    \end{equation}
where $\mathcal{F}_{Z}$ denotes the $\sigma$-algebra generated by $Z$. Since $D_{KL}(Z \| Z +  \varepsilon s^{d/2}) = \varepsilon^2 s^d/ 2$, combining \eqref{e:entbound3}--\eqref{e:entbound4} gives the result.

The general case follows by averaging over any auxiliary randomness in the algorithm, since by Jensen's inequality $\E_\ell[ \sqrt{  \E_\ell[N_{\mathcal{S}'} | \mathcal{G} ] } ] \le \E_\ell[ \sqrt{ N_{\mathcal{S}'}  } ]$ for any sub-$\sigma$-algebra $\mathcal{G}$.

\begin{remark}
To prove the Bernoulli version in Proposition \ref{p:ebber2} one can follow the same steps leading to \eqref{e:entbound3}--\eqref{e:entbound4}, but replace $D_{KL}(Z \| Z +  \varepsilon s^{d/2}) = \varepsilon^2 s^d/ 2$ with the bound (recalling \eqref{e:klber})
\begin{align*}
  D_{KL}(\Ber(p) \| \Ber(q) )    =  \int_q^p  \frac{p-s}{s(1-s)} ds  & \le \max\Big\{   \frac{1}{p(1-p)}, \frac{1}{q(1-q)}   \Big\}  \int_q^p (p-s) ds  \\
 & = \max\Big\{   \frac{1}{2p(1-p)}, \frac{1}{2q(1-q)}   \Big\} (p-q)^2   .
 \end{align*}
\end{remark}
\begin{remark}
\label{r:alt}
For comparison we sketch an alternative approach which is similar to arguments in \cite{os07,ss11,vn20}; this recovers the differential form of Proposition \ref{p:entbound2}, i.e.\ 
\begin{equation}
\label{e:ubalt}
\Big|  \frac{d}{d \varepsilon}  \mathbb{P}_\ell\Big[ f + \varepsilon \sum_{S \in \mathcal{S}' } q \star \id_S \in A  \Big] \Big|_{\varepsilon=0} \Big|  \le  s^{d/2}  \sqrt{   \P_\ell[A]  \E_\ell  |N_{\mathcal{S}'}|} ,
 \end{equation}
 but is not sufficient for all of our results.
 
 Recall the representation
\[  f_S(\cdot) + \ell + \varepsilon (q \star \id_S)(\cdot) \stackrel{d}{=} \frac{(Z_S +  \varepsilon s^{d/2})(q \star \id_S)(\cdot)}{s^{d/2}} + \ell + g_S(\cdot) . \]
 Then the Cameron-Martin formula applied to $(Z_S, g_S)_S$ gives that
\begin{equation}
\label{e:russocov}
  \frac{d}{d \varepsilon}  \mathbb{P}_\ell\Big[ f + \varepsilon \sum_{S \in \mathcal{S}' } q \star \id_S \in A  \Big]  \Big|_{\varepsilon=0}   =   s^{d/2} \sum_{S \in \mathcal{S}'} \cov_\ell[\id_A,Z_S]  .
\end{equation}
Decompose
\[ \sum_{S \in \mathcal{S}'} \cov_\ell[\id_A,Z_S]  = \sum_{S \in \mathcal{S}'}   \cov_\ell [\one_A \one_{\rev(S)}, Z_S ] +\sum_{S \in \mathcal{S}'}  \cov_\ell [\one_A \one_{\rev(S)^\compl}, Z_S ] . \]
 One can check that $\one_A \one_{\rev(S)^\compl}$ is independent of $Z_S$ and so the second sum vanishes. Hence~\eqref{e:russocov} is in absolute value at most
\[  \Big| \sum_{S \in \mathcal{S}'}  \cov_p[ \one_A \one_{\rev(S)},Z_S ]  \Big| \leq  s^{d/2}  \sqrt{ \P_\ell[A] \E_\ell \Big[ \Big( \sum_{S \in \mathcal{S}'}  \one_{\rev(S)} Z_S \Big)^2\Big] }  \]
where we used the the Cauchy-Schwartz inequality. For $S$ and $S'$ introduce the event 
\[ \rev(S,S'):=\rev(S)\cap\rev(S')\cap\{S \text{ is revealed before }S'\}. \]
Again one checks that, for $S\neq S'$, $\one_{\rev(S,S')}Z_S$ is independent of $Z_{S'}$, and also that $\one_{\rev(S)}$ is independent of $Z_S$. Hence 
\[ \E_\ell \Big[\Big(\sum_{S \in \mathcal{S}'}  \one_{\rev(S)} Z_S \Big)^2\Big]   = \sum_{S \in \mathcal{S}'} \E_\ell \left[\one_{\rev(S)} Z_S^2\right]  = \sum_{S \in \mathcal{S}'}   \E_\ell[ \id_{\rev(S)} ] = \E_\ell [N_{\mathcal{S}'}] , \]
where the first equality used that off-diagonal terms are zero by independence, and the second equality again used independence and that $\E[Z_S^2]=1$. Combining the above gives \eqref{e:ubalt}.
\end{remark}

%%%%%%%%%%%
%%%%%%%%%%%

\medskip
\section{Finite-range models}
\label{s:bou}

In this section we prove our results on the one-arm exponent for Gaussian models with finite-range dependence. Throughout we suppose that Assumption \ref{a:gf} and \eqref{a:pos1} holds, and we let $r  >  0$ be such that $q$ is supported on $\Lambda_r$.

\smallskip
Let us begin by introducing notation for connection events. For $k, R> 0$, define the box $B_k(R):=  [-R, R] \times [-kR, kR]^{d-1}  \subset \X^d$, and the `box-crossing event'
\[  \textrm{Cross}_k(R) := \Big\{ \{-R\} \times [-kR,kR]^{d-1}  \stackrel{B_k(R)}{\longleftrightarrow} \{R\} \times[kR,kR]^{d-1}  \Big\}  \]
where, for $E \subset \X^d$,
\[ \{ A  \stackrel{E}{\longleftrightarrow} B \} :=  \{ \text{there exists a path in $\{f \ge 0\} \cap E$ that intersects $A$ and $B$}\}  . \]
For $R \ge 0$, define the \textit{one-arm event}
\[A_1(r,R):=\{ \Lambda_r \longleftrightarrow \partial \Lambda_R\}  \ , \quad A_1(R) = A_1(1,R) . \]
In the self-dual case that $f$ is continuous and $d=2$, we also introduce the one-arm event for \textit{level lines} mentioned in Remark~\ref{r:larm}, analogous to a \textit{polychromatic two-arm event} in Bernoulli percolation. This is defined as
\[ A_2(r,R):=\{ \Lambda_r   \doublepath \partial \Lambda_R\} \ , \quad A_2(R) = A_2(1,R) ,\]
where 
\begin{align*}
\{A   \stackrel{E}{\doublepath} B\} &= \{ \text{there exists a path in $\{f  = 0\} \cap E$ that intersects $\Lambda_r$ and $\partial \Lambda_R$}\}  \\
& = \{A    \stackrel{E}{\longleftrightarrow}B  \}\cap \{ \text{there exists a path in $\{f  \le 0\} \cap E$ that intersects $A$ and $B$}\}  , 
\end{align*}
and $\{A \doublepath B \} = \{A  \stackrel{\X^d}{\doublepath} B \}$. We make the elementary observation that, if the field is positively-correlated, \begin{equation}
\label{e:eta1eta2}
   \P_{\ell_c}[ A_2(r, R)] \le \P_{\ell_c}[A_1(r, R)]^2 .     \end{equation} 
   To see this, note that by the FKG inequality
\[ \mathbb{P}_\ell[ A_2(r,R) ] \le   \mathbb{P}_\ell[ A_1(r,R) ] \, \mathbb{P}_\ell[  \{ \text{there exists a path in $\{ f \le 0 \}$ that intersects $\Lambda_r$ and $\partial \Lambda_R$} \}] . \]
Moreover, by self-duality at $\ell_c = 0$,
 \[    \mathbb{P}_{\ell_c}[  \{ \text{there exists a path in  $\{f \le 0\}$ that intersects $\Lambda_r$ and $\partial \Lambda_R$} \}] =  \mathbb{P}_{\ell_c}[ A_1(r,R) ]  ,\]
 from which \eqref{e:eta1eta2} follows.
  
\smallskip
We now state three intermediate propositions which together will prove our results. The bound $\eta_1 \le d/3$ follows by combining the following inequality with the mean-field bound \eqref{a:mflb}:

\begin{proposition}
\label{p:ubgf1}
For $\ell \le \ell'$ and $R \ge r \ge 1$,
\[ \mathbb{P}_{\ell'}[ A_1(R)] - \mathbb{P}_\ell[ A_1(R)]   \le \frac{r^{d/2}(\ell'-\ell)}{\int q}  \sqrt{  \mathbb{P}_{\ell'}[A_1(R)]   \! \! \sum_{v \in r \Z^d \cap \Lambda_{R+2r}} \! \! \P_\ell[ \Lambda_1 \longleftrightarrow v + \Lambda_{2r} ]}  . \]
\end{proposition}

\noindent For the bound $\eta_1 \le 1/3$ in the self-dual case we rely instead on the following inequalities. Recall the Dini derivatives in \eqref{e:dini}.

\begin{proposition}
\label{p:ubgf2}
For every $k \ge 1$ there exists $c = c(k)>0$ such that, for $\ell \in \mathbb{R}$ and $R \ge 4r > 0$,
\[ \frac{d^+}{d \ell} \mathbb{P}_\ell[ \textrm{Cross}_k(R)] \leq  \frac{c   R^{d/2}}{\int q} \begin{cases}    \sqrt{ \P_\ell[A_2(2r,R-2r)] }  &  \text{ if and $d=2$,} \\ 
   \sqrt{ \P_\ell[A_1(2r,R-2r)] } &   \text{ in general.}\end{cases}\]
\end{proposition}

\begin{proposition}
\label{p:lbgf}
Suppose $f$ is continuous, $d=2$, and \eqref{a:pos2} holds. Then for every $k \ge 1$ there exists $c = c(k)>0$ such that, for $\ell \in \R$ and $R \ge 8r > 0$,
\[  \frac{d^-}{d \ell} \mathbb{P}_\ell[ \textrm{Cross}_k(R)] \geq   \frac{c}{\|q\|_2} \frac{ \P_\ell[\textrm{Cross}_{1/(8k)}(kR)]^4 \big(1 - \P_\ell[\textrm{Cross}_{8k}(R/8)] \big)^2 }{  \P_\ell[A_2(2r,R-2r)]} . \]
\end{proposition}

%\begin{remark}
%\label{r:simpler}
%One can replace the right-hand side of \eqref{e:lbgf2} with the (perhaps simpler) expression
%  \[ \frac{c}{ \|q\|_2 }\frac{ \P_\ell[\textrm{Cross}_k(R)]\big(1- \P_\ell[\textrm{Cross}_k(R)]\big) }{\frac{r}{R}  \sum_{i=2}^{R/r}\P_\ell[A_2(2r,ir)] }\]
%  which holds without \eqref{a:pos2}; see Proposition \ref{p:lbgf2}. While this suffices to prove $\eta_1 \le 1/3$, it does not yield the pointwise bound in Theorem \ref{t:bou}.
%\end{remark}

\noindent We prove Propositions \ref{p:ubgf1}--\ref{p:lbgf} later in the section; for now we establish Theorems \ref{t:bou} and~\ref{t:oab2}. 

\smallskip
Let us first state some auxiliary results on general $f = q \star W$ satisfying Assumption \ref{a:gf}; these are rather standard, especially for finite-range models, but we give details on their proof at the end of the section. 

\begin{lemma}
\label{l:classicalgf} 
Suppose $f$ satisfies Assumption \ref{a:gf}. 
\begin{enumerate}
\item There exists $\delta > 0$ and $\ell' = \ell'(R) \le \ell_c$ such that, for $R \ge 1$,
\[ \mathbb{P}_{\ell'}[\textrm{Cross}_5(R)] =  \delta .  \]
\item (RSW estimates) Suppose that $f$ is continuous, $d=2$, and \eqref{a:pos2} holds. Then for every $k \ge 1$ there exists $\delta > 0$ such that, for $R \ge 1$,
\[ \P_{\ell_c}[\textit{Cross}_k(R)] \in (\delta, 1-\delta) .\]
\item For every $k \ge 1$ and $R \ge r > 0$, $\textrm{Cross}_k(R)$ and $A_1(r, R)$ are continuity events.
\end{enumerate}
\end{lemma} 

\begin{proof}[Proof of Theorem \ref{t:bou}]
In the proof $c > 0$ are constants that depend only on $f$ and may change from line to line. Without loss of generality we assume that $r \ge 1$.

We begin by establishing the bound $\eta_1 \le d/3$. We may assume that $\mathbb{P}_{\ell_c}[ A_1(R) ]  \to 0$ as $R \to \infty$ since otherwise $\eta_1 = 0$. Define, for sufficiently large $R$,
\begin{equation}
    \label{e:ellprime}
 \ell'(R) = \inf\{ \ell > \ell_c :  \P_\ell[A_1( R)] = 2 \P_{\ell_c}[A_1( R) ] \} \in [\ell_c,\infty) ,
 \end{equation}
which exists by the continuity of $\ell \mapsto \P_\ell[A_2(R)]$ (third item of Lemma \ref{l:classicalgf}), and since $\P_{\ell_c}[A_1( R) ] > 0$ and
\[ \P_\ell[A_1( R)]  \ge  \P\big[ \inf_{x \in \Lambda_R} f(x)  \ge -\ell \big] \to 1   \]
as $\ell \to \infty$. Note that $\ell'(R) \to \ell_c$ as $R \to \infty$,  since if instead $\limsup_{R \to \infty}\ell'(R) > \ell'' > \ell_c$ then
 \[ \limsup_{R \to \infty} \P_{\ell_c}[A_1(R)]  \ge \limsup_{R \to \infty} \P_{\ell'}[A_1(R)]/2 \ge \theta(\ell'')/2 > 0  .\]

By the mean-field bound \eqref{a:mflb}, for sufficiently large $R$,
\begin{align}
\label{e:mfbuse}
   \P_{\ell_c}[A_1(R)]/2 =  \P_{\ell'}[A_1(R)]/4 \ge \theta(\ell')/4 \ge c(\ell'-\ell_c)  . 
\end{align}
Now apply Proposition \ref{p:ubgf1} to $\ell = \ell_c$ and $\ell'$; this yields
\[  \mathbb{P}_{\ell_c} [A_1(R) ]  =   \mathbb{P}_{\ell'} [A_1(R) ]  - \mathbb{P}_{\ell_c} [A_1(R) ]   \le  c (\ell'-\ell_c)  \sqrt{ \mathbb{P}_{\ell_c}[A_1(R)]   \! \! \sum_{v \in  r\Z^d \cap \Lambda_{R+2r}} \! \! \P_{\ell_c}[ \Lambda_1 \longleftrightarrow v + \Lambda_{2r} ]}   \]
 for large $R$. Combining with \eqref{e:mfbuse}, we deduce that
\begin{equation}
\label{e:mainineq}
\mathbb{P}_{\ell_c}[A_1(R)]    \! \! \sum_{v \in  r \Z^d \cap \Lambda_{R+2r}} \! \! \P_{\ell_c}[ \Lambda_1 \longleftrightarrow v + \Lambda_{2r} ]    \ge c   
\end{equation}
 for all $R \ge 1$. 
 
\smallskip
We now show that $\eta_1 \le d/3$ follows from \eqref{e:mainineq}. If $\eta_1 = 0$ there is nothing to prove, so assume $\eta_1 > 0$ and fix $\eta^\ast \in (0, \eta_1)$. Then by the definition of $\eta_1$
\begin{equation}
\label{e:etaast}
\mathbb{P}_{\ell_c}[ A_1(R) ]  \le R^{- \eta^\ast}  
\end{equation}
for large $R$. Observe next that, for $|x|_\infty \ge 18 r $, the event $\Lambda_1 \longleftrightarrow x + \Lambda_{2r}$ implies the occurrence of the events
\[  \{   A_1(|x|_\infty/3)  \}   \qquad \text{and} \qquad \{  x + A_1(6r, |x|_\infty/3)  \}  , \]
which are measurable with respect to disjoint domains separated by sup-distance at least $6r$. Since $q$ is supported on $\Lambda_r$, these events are independent. Hence for large $R$ and $|x|_\infty \ge 18r$ this implies
\begin{align*}
   \P_{\ell_c}[   \Lambda_1 \longleftrightarrow x + \Lambda_{2r} ]    \le \P_{\ell_c}[  A_1( |x|_\infty/3) ] \P_{\ell_c}[  A_1(6r, |x|_\infty/3) ] \le c    |x|_\infty^{-2\eta^\ast}   
   \end{align*}
where we used the union bound and  \eqref{e:etaast} in the second inequality. Then, for large $R$, by an integral comparison
  \begin{align*}
   \sum_{v \in  r\Z^d \cap \Lambda_{R+2r}} \! \! \P_{\ell_c}[   \Lambda_1 \longleftrightarrow v + \Lambda_{2r} ]  \le  c  + c \sum_{v \in  r\Z^d \cap \Lambda_{R+2r} \setminus \{0\} } \! \!   |v|_\infty^{-2\eta^\ast}   \le c + c    \max\{  R^{d-2 \eta^\ast } (\log R) ) , 1\}   
    \end{align*}
    Combining with \eqref{e:mainineq}, for large $R$,
    \begin{equation}
        \label{e:eta1con}
    c \le   R^{- \eta^\ast}   \Big( 1 +   \max\{  R^{d-2 \eta^\ast } (\log R ) , 1\}   \Big).   
     \end{equation}
    This implies that $\eta^{\ast} \le d/3$, and since $\eta^\ast < \eta_1$ was arbitrary, we deduce that $\eta_1 \le d/3$.

\smallskip
We now turn to the self-dual case. By the RSW estimates in Lemma \ref{l:classicalgf} and Propositions \ref{p:ubgf2} and \ref{p:lbgf}, for large $R$,
\[  c  \mathbb{P}_{\ell_c}[ A_2(2r,R-2r )]^{-1}   \le  \frac{d}{d \ell} \mathbb{P}_\ell[  \textrm{Cross}_1(R)] \Big|_{\ell= \ell_c}    \le   c  R \sqrt{\mathbb{P}_{\ell_c}[ A_2(2r,R-2r ) ]}  \]
which gives $ \mathbb{P}_{\ell_c}[ A_2(2r, R-2r )]   \ge  c R^{-2/3}$ for large $R$. Applying the union bound, we deduce that
\[  \mathbb{P}_{\ell_c}[ A_2( R )]   \ge  c R^{-2/3} \]
which, together with \eqref{e:eta1eta2}, gives the result.
\end{proof}

\begin{proof}[Proof of Theorem \ref{t:oab2}]
Recall that we only consider the case $\X = \Z$ and $q(x) = \id_{x=0}$, so that we may take $r = 1$. In the proof $c > 0$ are constants that depend only on the dimension and may change from line to line, and $o(1)$ denotes a quantity that decays to zero as $R \to \infty$. 

\smallskip We begin with the bounds $\eta_1 \le d / (2/\beta+1)$ and $\bar{\eta}_1 \le (2-\eta)/(2/\beta-1)$ which require only a slight change to the argument used to prove $\eta_1 \le d/3$ above. Let $\ell'(R) \to \ell_c$ be defined as in~\eqref{e:ellprime}. By the definition of the exponent $\beta$, one can replace \eqref{e:mfbuse} with 
\[   \P_{\ell_c}[A_1(R)] \ge \theta(\ell')/2 \ge c(\ell'-\ell_c)^{\beta + o(1)} , \]
 which gives, in place of \eqref{e:mainineq},
\begin{equation}
\label{e:mainineq2}
  \mathbb{P}_{\ell_c}[A_1(R)]^{\frac{2}{\beta}- 1 + o(1)}   \! \! \sum_{v \in  \Z^d \cap \Lambda_{R+2}} \! \! \P_{\ell_c}[ \Lambda_1 \longleftrightarrow v + \Lambda_2 ]    \ge c  .
 \end{equation}
Then in place of \eqref{e:eta1con} we have, for any $\eta^\ast \in (0, \eta_1)$ and large $R$ 
\[ c \le R^{-\eta^\ast (\frac{2}{\beta}- 1) + o(1) } \Big( 1 +  \max\{ R^{d-2\eta^\ast} (\log R), 1\} \Big)    ,    \]
which implies $ \eta_1 \le d / (2/\beta+1)$. On the other hand, by the definition of the exponent $\eta$,
\[ \sum_{v \in \Lambda_R} \P_{\ell_c}[ 0 \longleftrightarrow v ]   = R^{2 - \eta + o(1) } , \]
which by \eqref{e:mainineq2} and the union bound implies 
\[  \mathbb{P}_{\ell_c}[A_1(R)] \ge R^{-(2-\eta)/(2/\beta-1)  + o(1) }  \]
and hence $\bar{\eta}_1 \le (2-\eta)/(2/\beta-1)$. 

\smallskip
To prove the remaining bounds we use the fact that, by a super-multiplicativity argument (see \cite[Section 6.2]{gri99}),\begin{equation}
\label{e:tpexp}
 \mathbb{P}_\ell[ 0 \longleftrightarrow v ] \le e^{-  |v|_\infty /\xi(p) }  
 \end{equation}
for all $\ell < \ell_c$ and $v \in \Z^d$. We also recall the standard facts \cite[Theorem 6.14]{gri99} that $\xi(\ell)$ is continuous, strictly increasing, and $\xi(\ell)\to \infty$ as $\ell \uparrow \ell_c$.

\smallskip
Let $C > 0$ be a constant to be fixed later, and for $R$ sufficiently large, let $\ell'' = \ell''(R) \uparrow \ell_c$ be such that $R = C \xi(\ell'') \log R$. Since we have the a priori bound $\P_{\ell_c}[A_1(R)] \ge c R^{-(d-1)/2}$ \cite{ham57,tas87}, we can take $C > 0$ sufficiently large so that, by \eqref{e:tpexp} and the union bound, 
\[  \mathbb{P}_{\ell''}[A_1(R) ]  \le c R^{d-1} R^{-C}  \le  \P_{\ell_c}[A_1(R)] / 2  \]
for large $R$.  Then applying Proposition \ref{p:ubgf1} to $\ell = \ell''$ and $\ell' = \ell_c$ gives, for large $R$,
\[  \mathbb{P}_{\ell_c} [A_1(R) ] /2  \le   \mathbb{P}_{\ell_c} [A_1(R) ]  - \mathbb{P}_{\ell''} [A_1(R) ]   \le  c (\ell_c-\ell'') \sqrt{ \mathbb{P}_{\ell_c}[A_1(R)]  \! \! \sum_{v \in \Z^d \cap \Lambda_{R+2}} \! \! \P_{\ell''}[ \Lambda_1 \longleftrightarrow v + \Lambda_2 ]}   \]
or, equivalently,
\begin{equation}
\label{e:mainineq3}
\P_{\ell_c}[A_1(R)]  \le c (\ell_c-\ell'')^2  \mathbb{P}_{\ell_c}[A_1(R)]  \! \! \sum_{v \in \Z^d \cap \Lambda_{R+2}} \! \! \P_{\ell''}[ \Lambda_1 \longleftrightarrow v + \Lambda_2 ].
\end{equation}
Since $ \sum_{v \in \Z^d \cap \Lambda_{R+2}} \! \! \P_{\ell''}[ \Lambda_1 \longleftrightarrow v + \Lambda_2 ] \le c \chi(\ell'')$, and by the definition of the exponents $\nu$ and $\gamma$, this implies
\[\P_{\ell_c}[A_1(R)]  \le c  (\ell_c-\ell'')^2  \chi(p')  \le c     \xi(\ell'')^{-2/\nu + o(1) } \xi(\ell'')^{\gamma/\nu + o(1)} = R^{-(2-\gamma)/\nu + o(1)}   \]
for large $R$, which implies $\eta_1 \ge (2-\gamma)/\nu$.

\smallskip
Finally, let $\delta > 0$ be such that $\mathbb{P}_{\ell_c}[\textrm{Cross}_5(R)]  \ge  \delta$ for large $R$ (possible by the first statement of Lemma~\ref{l:classicalgf}), and again let $\ell'' = \ell''(R) \uparrow \ell_c$ be such that $R = C \xi(\ell'') \log \xi(\ell'')$. Then
\[  \mathbb{P}_{p'}[\textrm{Cross}_5(R) ] \le  \delta/2 \]
for large $R$, and we deduce that there exists $\ell''' \in (\ell'',\ell_c)$ such that
\[    \frac{d}{d \ell} \mathbb{P}_\ell[ \textrm{Cross}_5(R)] \Big|_{\ell= \ell'''}  \ge \frac{\delta/2}{\ell_c-\ell''} .   \]
On the other hand, by Proposition \ref{p:ubgf2} and monotonicity in $\ell$,
\[  \frac{d}{d \ell} \mathbb{P}_\ell[  \textrm{Cross}_5(R) ] \Big|_{\ell= \ell'''} \le   c  R^{d/2} \sqrt{ \P_{\ell_c}[A_1(R)] } \]
and hence
\[     (\ell_c-\ell'')^2 R^d \, \P_{\ell_c}[A_1(R)]     \ge  c  \]
for large $R$. By the definition of the exponent $\nu$, this implies
\[  \xi(\ell'')^{-2/\nu + o(1)} R^d \P_{\ell_c}[A_1(R)]  =  R^{d-2/\nu + o(1)}  \P_{\ell_c}[A_1(R)]    \ge c \]
for large $R$, which implies that $\bar{\eta}_1 \le d - 2/\nu$.
\end{proof}

\begin{remark}
\label{r:high2}
As mentioned in Remark \ref{r:high}, by combining the high-dimensional bounds~\cite{har90, har08}
\[   \xi(\ell) \le c(\ell_c-\ell)^{-1/2}   \qquad \text{and} \qquad \P_{\ell_c}[ 0 \longleftrightarrow v ] \le c |v|_{\infty}^{-d+2}   \]
with \eqref{e:mainineq} and \eqref{e:mainineq3}, one arrives at a quantitative version of Corollary \ref{c:oab}, namely the bounds
\[  c_1 R^{-2}  \le \P_{\ell_c}[A_1(R)] \le c_2 R^{-2} (\log R)^4. \]
\end{remark}

\subsection{Randomised algorithms}

To prove Propositions \ref{p:ubgf1}--\ref{p:lbgf} we work with well-chosen randomised algorithms that determine the events $A_1(R)$ and $\textrm{Cross}_k(R)$.
 
 \smallskip
 Recall the box $B_k(R) = [-R,R] \times [-kR,kR]^{d-1} \subset \X^d$, and define its \textit{right half} $B_k^+(R) := [0, R] \times [-kR, kR]^{d-1}  \subset \X^d$. If $d=2$, define its \textit{top-right quarter} $B_k^\dagger(R) := [0, R] \times [0,kR]  \subset \X^d$. We denote by $d_\infty(A,B)$ the distance between $A,B \subset \X^d$ with respect to the sup-norm.

\begin{lemma}\label{l:algo}
For every $\ell \in \R$ and $R \ge r$ there is an algorithm in $\mathcal{A}_r$ determining $A_1(R)$ such that, under $\P_\ell$,
\[    \sum_{S \in \mathcal{S}_r} \textrm{Rev}(S) \le  \sum_{v \in r \Z^d \cap \Lambda_{R+2r}} \! \! \P_\ell[ \Lambda_1 \longleftrightarrow v + \Lambda_{2r} ] .  \]
Moreover for every $k \ge 1$, $\ell \in \R$, and $ R \ge 4r > 0$, there are algorithms in $\mathcal{A}_r$ determining $\textrm{Cross}_k(R)$ such that, under $\mathbb{P}_\ell$, these algorithms satisfy respectively
\[ \max_{S \in \mathcal{S}_r  : d_\infty(S, B_k^+(R)) \le r }  \! \! \textrm{Rev}(S) \le \P_\ell[A_1(2r, R-2r)]  ,\]
and, if $d=2$,
\[  \max_{S \in \mathcal{S}_r  : d_\infty(S,B_k^\dagger(R) ) \le r}  \! \! \textrm{Rev}(S) \le  \P_\ell[A_2(2r, R-2r)] . \]
\end{lemma}

\noindent The algorithms in Lemma \ref{l:algo} are classical for Bernoulli percolation, however since we work in a rather more general setting we provide details at the end of the section.

\subsection{Proof of Propositions \ref{p:ubgf1}--\ref{p:ubgf2}}

To prove Propositions \ref{p:ubgf1}--\ref{p:ubgf2} we combine the entropic bound with the algorithms in Lemma \ref{l:algo}:

\begin{proof}[Proof of Proposition \ref{p:ubgf1}]
This follows directly from Proposition \ref{p:entbound} by considering the algorithm in Lemma~\ref{l:algo} that determines $A_1(R)$ such that  
\begin{equation*}  \E_\ell [N]  =   \sum_{S \in \mathcal{S}_r} \textrm{Rev}(S) \le  \sum_{v \in r \Z^d \cap \Lambda_{R+2r}} \! \! \P_\ell[ \Lambda_1 \longleftrightarrow v + \Lambda_{2r} ] . \qedhere
\end{equation*}
\end{proof}

\begin{proof}[Proof of Proposition \ref{p:ubgf2}]
We begin with the general case. We first partition the set of boxes $\{ S \in \mathcal{S}_r :  d_\infty(S, B_k(R)) \le r \}$ that cover $B_k(R)$ into the disjoint sets  
\[ \mathcal{S}'_1 = \{ S \in \mathcal{S}_r :  d_\infty(S, B_k^+(R)) \le r \} \qquad \text{and} \qquad \mathcal{S}'_2= \{ S \in \mathcal{S}_r :  d_\infty(S, B_k(R)) \le r \} \setminus \mathcal{S}'_1 \]
Note that $\mathcal{S}'_1 $ and $ \mathcal{S}'_2$ correspond roughly to boxes which cover, respectively, the right-half $B_k^+(R)$ and its complement $B_k(R) \setminus B_k^+(R)$, except that we enforce disjointness (see Remark \ref{r:disjoint} for an explanation) so we do not have exact reflective symmetry. However the reflection of $\mathcal{S}'_2$ in the hyperplane $\{0\} \times \R^{d-1}$ is contained in $\mathcal{S}'_1$.

\smallskip
By disjointness and since $q$ is supported on $\Lambda_r$, for every $x \in B_k(R)$ we have
\[   \sum_{i = 1,2} \sum_{S \in \mathcal{S}'_i} (q \star \id_S)(x)  =    \sum_{S \in \mathcal{S}'_1 \cup \mathcal{S}'_2 } (q \star \id_S)(x)  =   (q \star \id)(x)   =  \int q .  \]
Then by the multivariate chain rule for Dini derivatives 
\begin{align*}
\frac{\partial^+}{\partial \ell} \P_\ell \big[\textrm{Cross}_k(R) \big] = &   \frac{1}{\int q}  \frac{\partial^+}{\partial \varepsilon}  \P_\ell \big[f+   \varepsilon  \sum_{i = 1,2} \sum_{S \in \mathcal{S}'_i}  q \star \id_S   \in \textrm{Cross}_k(R) \big] \Big|_{\varepsilon = 0}  \\
&  \le \frac{1}{\int q} \sum_{i=1,2}   \frac{\partial^+}{\partial \varepsilon}  \P_\ell \big[f+   \varepsilon \sum_{S \in \mathcal{S}'_i  } q \star \id_S   \in \textrm{Cross}_k(R) \big] \Big|_{\varepsilon = 0}   .
\end{align*}
Now consider the algorithm in Lemma \ref{l:algo} that determines $\textrm{Cross}_k(R)$ such that, under $\P_\ell$,
\[  \max_{S \in \mathcal{S}'_1 }   \textrm{Rev}(S) \le \P_\ell[A_1(2r, R-2r)]  .   \]
By reflective symmetry, there is also an algorithm determining  $\textrm{Cross}_k(R)$  such that, under $\P_\ell$,
\[  \max_{S \in \mathcal{S}_2'}   \textrm{Rev}(S) \le \P_\ell[A_1(2r, R-2r)]  .   \]
Since also $\max_{i=1,2} |\mathcal{S}'_i|  \le c_1 (R/ r)^d$ for some $c_1  > 0$ depending only on $k$ and $d$, applying Proposition \ref{p:entbound2} gives
\[ \frac{\partial^+}{\partial \ell} \P_\ell \big[\textrm{Cross}_k(R) \big]   \le  \frac{r^{d/2}}{2 \int q}  \sqrt{c_1 (R/r)^d   \P_\ell[A_1(2r, R-2r)]  }  =   \frac{c_2 R^{d/2} }{\int q}  \sqrt{ \P_\ell[A_1(2r, R-2r)]  }  \]
for some $c_2 = c_2(k,d) > 0$, as required.

\smallskip
For $d=2$ we consider the top-right quadrant $B_k^\dagger(R)$ and the algorithm in Lemma \ref{l:algo} that determines $\textrm{Cross}_k(R)$ such that
\[  \max_{ \{ S \in \mathcal{S}_r :  d_\infty(S, B_k^\dagger(R)) \le r \} }   \textrm{Rev}(S) \le \P_\ell[A_2(2r, R-2r)]  .   \]
Then a similar argument to in the previous case (except partitioning $\{ S \in \mathcal{S}_r :  d_\infty(S, B_k(R)) \le r \}$ into $\cup_{i = 1,\ldots,4} \mathcal{S}'_i$ into four disjoint sets that approximate the four quadrants of $B_k^+(R)$ and using reflective symmetry in both axes) yields the result.
\end{proof}

\begin{remark}
\label{r:disjoint}
Since we do not assume $q \ge 0$, it is not necessarily true that
\[    \frac{\partial^+}{\partial \varepsilon}  \P_\ell \big[f+   \varepsilon  q \star \id_S   \in \textrm{Cross}_k(R) \big]  \ge 0 \]
for every $S \in \mathcal{S}_r$. Hence in the proof of Proposition \ref{p:ubgf2} it was crucial that we partitioned $\{ S \in \mathcal{S}_r :  d_\infty(S, B_k(R)) \le r \}$ disjointly into $\cup_i \mathcal{S}'_i$, since otherwise we could not deduce that
\[ \frac{\partial^+}{\partial \ell} \P_\ell \big[\textrm{Cross}_k(R) \big] \le   \frac{1}{\int q}  \frac{\partial^+}{\partial \varepsilon}  \P_\ell \big[f+   \varepsilon  \sum_{i=1,2} \sum_{S \in \mathcal{S}'_i}  q \star \id_S   \in \textrm{Cross}_k(R) \big] \Big|_{\varepsilon = 0}  . \]
\end{remark}

\subsection{Proof of Proposition \ref{p:lbgf}}

Before proving Proposition \ref{p:lbgf} we first state a generalisation, which follows by combining our Russo-type inequality in Proposition \ref{p:russo} with (a conditional version of) the OSSS inequality.
 
 \begin{proposition}
\label{p:lbderiv}
There exists a $c > 0$ depending only on $d$ such that, for every $\ell \in \R$, increasing compactly supported continuity event $A$, $s > 0$, algorithm $\mathcal{A} \in \mathcal{A}_s$ that determines~$A$, and subset $\mathcal{S}' \subseteq\mathcal{S}_s$,
\begin{equation}
\label{e:lbderiv}
\frac{d^-}{d \ell} \mathbb{P}_\ell[A]   \ge \frac{c \min\{1,(s/r)^d \}}{\|q\|_2} \frac{\textrm{Var}_\ell [ \mathbb{P}_\ell[A | \mathcal{F}_{\mathcal{S}' } ] ] }{\max_{S \in \mathcal{S}' } \textrm{Rev}(S)} ,
\end{equation}
where $\mathcal{F}_{\mathcal{S}'}$ denotes the $\sigma$-algebra generated by $(f_S)_{S \in \mathcal{S}'}$, and the revealments $\textrm{Rev}(S)$ are under~$\P_\ell$.
\end{proposition}

\begin{proof}[Proof of Proposition \ref{p:lbderiv}]
We shall use the following version of the OSSS inequality (which applies to arbitrary compactly supported event $A$)
\begin{equation}
\label{e:OSSSext2}
  \Var_\ell \big[ \P_\ell[A  \: | \: \mathcal{F}_{\mathcal{S}'} ]  \big]  \le  \frac{1}{2} \sum_{S \in \mathcal{S}'} \textrm{Rev}(S) \Infl_{\ell;A}(S) \le    \frac{ \max_{S \in \mathcal{S}'} \textrm{Rev}(S)}{2} \sum_{S \in \mathcal{S}'} \Infl_{\ell;A}(S)
   \end{equation}
   where $\Infl_{\ell;A}(S)$ is defined as in \eqref{e:infl}.  Assuming \eqref{e:OSSSext2}, since $\Infl_{\ell;A}(S) \ge 0$ we have
\[   \sum_{S \in \mathcal{S}_s}   \Infl_{\ell;A}(S)  \ge \sum_{S \in \mathcal{S}'}   \Infl_{\ell;A}(S)  \ge \frac{2 \textrm{Var}_\ell[ \mathbb{P}_\ell[A | \mathcal{F}_{\mathcal{S}' } ] ] }{  \max_{S \in \mathcal{S}'} \textrm{Rev}(S) }  ,  \]
and combining with the Russo-type inequality in Proposition \ref{p:russo} yields the result.
   
 To derive \eqref{e:OSSSext2}, consider modifying the algorithm $\mathcal{A}$ to first reveal all $S \notin \mathcal{S}'$ and then proceed as in its original definition; this modification does not change revealment probabilities for $S \in \mathcal{S}'$. Applying the OSSS inequality \cite{osss05} to the vector $(\sum_{S \notin \mathcal{S}'} f_S, (f_S)_{S \in \mathcal{S}'})$ gives
\[  \Var_\ell[\one_A] \le  \frac{1}{2} \Big( \Infl_{\ell;A}(\mathcal{S'}^c) +   \sum_{S \in \mathcal{S}'} \textrm{Rev}(S) \Infl_{\ell;A}(S) \Big) \]
where $\Infl_{\ell;A}(\mathcal{S'}^c)$ is defined as
\[ \Infl_{\ell;A}( \mathcal{S'}^c) :=   \mathbb{P}_\ell \big[ \id_{\{f \in A\}} \neq \id_{\{f^{(\mathcal{S}^c)} \in A\}} \big] , \]
where $f^{(\mathcal{S}^c)}$ denotes the field $f = \sum_{S \in \mathcal{S}_s} f_S$ with $\sum_{S \notin \mathcal{S}'} f_S$ resampled independently. Since 
\begin{align*}
  \frac12 \Infl_{\ell;A}( \mathcal{S'}^c)  & =  \frac{1}{2} \E_p \big[ \P_p[ \text{the outcome of $A$ changes when $\sum_{S \notin \mathcal{S}' } f_S$ is resampled}  \: | \: \mathcal{F}_{\mathcal{S}'}  ] \big] \\
&=  \E_\ell \big[ \P_\ell[A  \: | \: \mathcal{F}_{\mathcal{S}'}] (1- \P_\ell[A  \: | \: \mathcal{F}_{\mathcal{S}'} ]) \big] = \E_\ell \big[ \Var_\ell[\id_A  \: | \: \mathcal{F}_{\mathcal{S}'} ]  \big] ,
\end{align*}
 combining with the law of total variance,
\[ \Var_\ell[\one_A]  -   \Infl_{\ell;A}( \mathcal{S'}^c) /2  =    \Var_\ell[\one_A]  - \E_\ell \big[ \Var_\ell[\id_A  \: | \: \mathcal{F}_{\mathcal{S}'} ]  \big]  =  \Var_\ell\big[ \P_\ell[A  \: | \: \mathcal{F}_{\mathcal{S}'} ]  \big]  \]
yields \eqref{e:OSSSext2}.
\end{proof}

Let us complete the proof of Proposition \ref{p:lbgf}:

\begin{proof}[Proof of Proposition \ref{p:lbgf}]
Recall the top-right quarter $B_k^\dagger(R)$ and set
\[ \mathcal{S}' = \{ S \in \mathcal{S}_r : d_\infty(S,B_k^\dagger(R)) \le r \} . \]
We claim
 \begin{equation}
 \label{e:var}
   \textrm{Var}_\ell \big[ \P_\ell [ \textrm{Cross}_k(R)  \: | \: \mathcal{F}_{\mathcal{S}'} ] ]  \ge  \P_\ell[  \textrm{Cross}_{1/(8k)}(kR) ]^4 (1- \P_\ell[  \textrm{Cross}_{8k}(R/8) ] )^2   .
   \end{equation}
Assuming \eqref{e:var}, the statement follows by applying Proposition \ref{p:lbderiv} (in the case $s=r$) to the algorithm in Lemma~\ref{l:algo} that determines $\textrm{Cross}_k(R)$ whose revealments on $\mathcal{S}'$ are bounded by $\P_\ell[A_2(2r, R-2r)]$.

\smallskip
To prove \eqref{e:var}, remark first that, for any event $A$ and sub-$\sigma$-algebra~$\mathcal{G}$, 
\begin{align}
\label{e:varbound}
  \textrm{Var} [ \P [ A | \mathcal{G} ] ]   =  \E [   (  \P [ A | \mathcal{G} ]  - \P[A] )^2 ]   & \ge    \sup_{A' \in \mathcal{G}}  \E [   (  \P [ A | \mathcal{G} ]  - \P[A] )^2 \, \id_{A'} ]  \\
  \nonumber & \ge \sup_{A' \in \mathcal{G}} \P[A'] ( \P[A | A' ] - \P[A]   )^2  
  \end{align}
where the second inequality is Jensen's. Hence it is enough to construct an event $A'$, measurable with respect to $\mathcal{F}_{\mathcal{S}'}$, such that $ \textrm{Cross}_k(R)$ becomes substantially more likely if $A'$ occurs (see Figure \ref{f:var} for an illustration).

\smallskip
 Define
\begin{align*}
& A' :=  \Big\{ \{R/4\} \times [R/4,R/2]  \stackrel{[R/4,R] \times [R/4,R/2]}{\longleftrightarrow} \{R\} \times[R/4,R/2]  \Big\} \\
&  \qquad \qquad  \cap  \Big\{ [R/4,R/2] \times \{R/4\}  \stackrel{[R/4,R/2] \times [R/4,kR]}{\longleftrightarrow} [R/4,R/2] \times \{kR\}  \Big\} ,  
\end{align*}
which is measurable with respect to $\mathcal{F}_{\mathcal{S}'}$. By the FKG inequality and symmetry (and an obvious event inclusion), $\P_\ell[A'] \ge \P_\ell[ \textrm{Cross}_{1/(8k)}(kR)]^2$. Define also the events
\begin{align*}
 B_1 & :=   \Big\{ \{-R\} \times [R/2,3R/4]  \stackrel{[-R,R/2] \times [R/2,3R/4]}{\longleftrightarrow} \{R/2\} \times[R/2,3R/4]  \Big\}   \\ 
B_2 &:=    \Big\{ \{3R/4\}  \times [-kR,kR]  \stackrel{[3R/4,R] \times [-kR, kR]}{\longleftrightarrow} \{R\} \times [-kR,kR]  \Big\}  
\end{align*}
which are defined on domains separated by a sup-norm distance at least $r$ (recall that $R \ge 8r$), and are translated copies of, respectively, $\textrm{Cross}_{1/6}(3R/2)$ and $\textrm{Cross}_{8k}(R/8)$. Finally, define
\[ C \! := \! \! \Big\{ \! \{-R\} \times [-kR, kR] \stackrel{B_k(R)}{\longleftrightarrow} \big( \{R/2\} \times [R/2,kR]  \big) \! \cup \! \big( [R/2,R] \times \{R/2\} \big) \! \cup \! \big(  \{R\} \times [-kR, R/2] \big) \!  \Big\}  \]
and observe (i) $\textrm{Cross}_k(R) \subseteq C$, (ii) on $A'$, $\textrm{Cross}_k(R) = C$, and (iii) $B_1 \cap B_2^\compl \subseteq C \setminus \textrm{Cross}_k(R)$. Hence
\begin{align*}
  \P_\ell [ \textrm{Cross}_k(R)  | A' ]  - \P_\ell [ \textrm{Cross}_k(R) ]  & = \P_\ell[C | A' ] - \P_\ell[ \textrm{Cross}_k(R) ]  \ge \P_\ell[C] -  \P_\ell[ \textrm{Cross}_k(R) ]   \\
  & = \P_\ell[ C \setminus \textrm{Cross}_k(R) ]  \ge \P_\ell[B_1 \cap B_2^\compl]   = \P_\ell[B_1] (1-\P_\ell[B_2])  \\
  & \ge \P_\ell[  \textrm{Cross}_{1/(8k)}(kR) ]  \big(1- \P_\ell[  \textrm{Cross}_{8k}(R/8) ] \big) ,
  \end{align*}
where the second step is by the FKG inequality, the penultimate step uses the independence of the events $B_1$ and $B_2$, and the final step is an obvious event inclusion. Applying \eqref{e:varbound} (with $A = \textrm{Cross}_k(R)$ and $\mathcal{G} = \mathcal{F}_{\mathcal{E}'}$) gives~\eqref{e:var}.
\begin{figure}
\centering
\includegraphics[height=6cm]{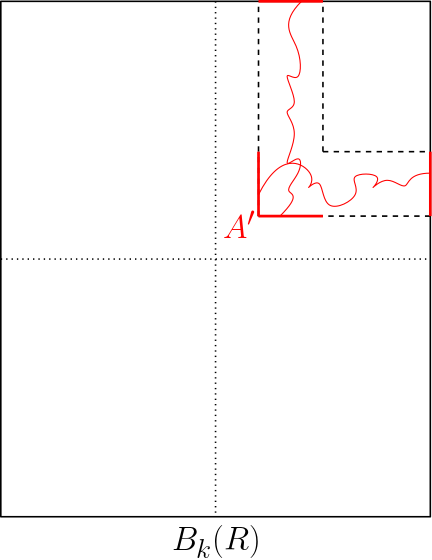}
\includegraphics[height=6cm]{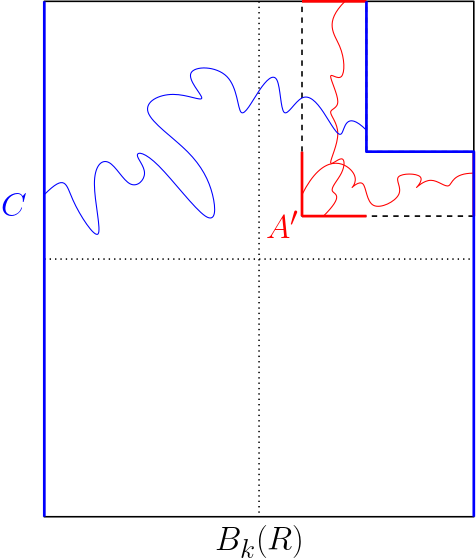}
\includegraphics[height=6cm]{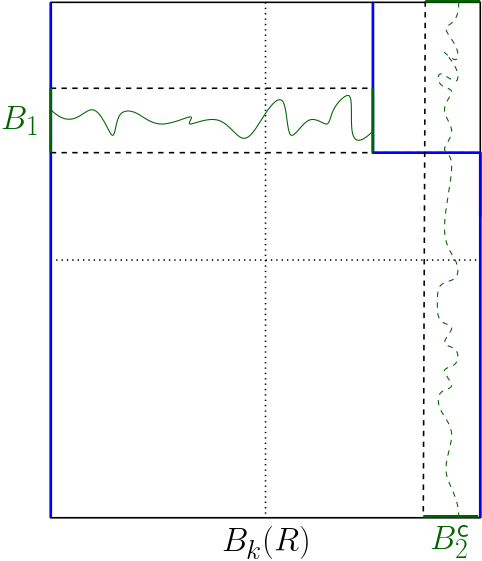}
\caption{An illustration of the proof of \eqref{e:var}. The first panel shows the event~$A'$. The second illustrates how, on $A'$, the event $C$ is equivalent to $\textrm{Cross}_k(R)$. The third shows how the crossing given by $B_1$, combined with the dual crossing given by $B_2^\compl$, realises $C$ but not $\textrm{Cross}_k(R)$.}
\label{f:var}
\end{figure}
\end{proof}

\begin{remark}
\label{r:genrev}
Similarly to in Section \ref{s:genbound}, combining Propositions \ref{p:entbound} and \ref{p:lbderiv} yields a general lower bound on the revealments of increasing events. We omit the proof, but the result is the following. Let $\ell \in \R$, let $R \ge r$, let $A$ be an increasing continuity event supported on $B(R)$, let $s > 0$, and let $\mathcal{A} \in \mathcal{A}_s$ be an algorithm determining $A$. Then there exists a $c > 0$ depending only on $d$ such that
\[   \max_{S \in \mathcal{S}_s} \textrm{Rev}(S) \ge   \frac{ c \big( (\int q / \|q\|_2 )\min\{1, (s/r)^d \}  \Var_\ell[ \id_A ]  \big)^{2/3} }{  \P_\ell[A]^{1/3} R^{d/3}}   \]
where the revealments $\textrm{Rev}(S)$ are under $\P_\ell$.

\smallskip
One can also prove a lower bound on $\max_{S \in \mathcal{S}'} \textrm{Rev}(S)$ for general $\mathcal{S}' \subset \mathcal{S}_s$, analogous to Proposition \ref{p:genrev2}, however in that case we would need $q \ge 0$ (for the same reason as explained in Remark \ref{r:disjoint} above) and also the refinement to Proposition \ref{p:russo} mentioned in Remark \ref{r:refine}.
\end{remark}

\subsection{Proof of Lemma \ref{l:classicalgf}}
For the first statement, it is enough to prove that
\begin{equation}
\label{e:boot}
 \liminf_{R \to \infty}  \mathbb{P}_{\ell_c}[\textrm{Cross}_5(R)] > 0  
\end{equation}
since then the result follows by the continuity of $\ell \mapsto \mathbb{P}_\ell[\textrm{Cross}_5(R)]$. By a classical bootstrapping argument \cite[Section 5.1]{kes82} (and Lemma \ref{trunc} below), there are $c_1, \varepsilon > 0$ such that
\begin{equation}
\label{e:boot2}
 \mathbb{P}_\ell[\textrm{Cross}_5(3R)] \le c_1 \big(  \mathbb{P}_\ell[\textrm{Cross}_5(R)]^2 + R^{-\varepsilon} \big)  
\end{equation}
for $\ell \in \R$ and $R$ sufficiently large. A consequence of \eqref{e:boot2} and the continuity of $\ell \mapsto \mathbb{P}_\ell[\textrm{Cross}_5(R)]$ is that  
\[ \liminf_{R \to \infty}  \mathbb{P}_\ell[\textrm{Cross}_5(R)]  < 1/c_1   \quad \implies \quad  \liminf_{R \to \infty}  \mathbb{P}_{\ell'}[\textrm{Cross}_5(R)]  = 0 \quad \text{for some } \ell' > \ell .\]
Covering the annulus $\Lambda_{5R} \setminus \Lambda_{3R}$ with $2d$ symmetric copies of $B_5(R)$, one can find a finite collection of copies $A_i$ of $\textrm{Cross}_5(R)$  such that $ \{\Lambda_1 \longleftrightarrow \infty\} \subseteq A_1(3R, 5R)  \subseteq \cup_i A_i $. Hence we also have
\[ \liminf_{R \to \infty}  \mathbb{P}_{\ell'}[\textrm{Cross}_5(R)] = 0 \quad \implies \quad \P_{\ell'}[\Lambda_1 \longleftrightarrow \infty] = 0  \quad \implies \quad \ell' \le \ell_c , \]
and so we deduce \eqref{e:boot}.

\smallskip
The RSW estimates in the second statement are a general property of stationary planar percolation models with positive associations and the symmetry in Assumption \ref{a:gf} \cite{kt20}.

\smallskip
For the third statement, consider first the case $\X = \Z$. Under Assumption \ref{a:gf} the covariance kernel $K$ is strictly positive definite, which implies that every compactly supported event is a continuity event. So let us assume $\X = \R$. Letting $g$ be a smooth function, since $f$ is $C^2$-smooth and $(f(x), \nabla f(x), \nabla^2 f(x))$ is non-degenerate, by \cite[Lemma 11.2.11]{at07}, for any $\ell\in\R$, almost surely there are no critical points of $f$ in the set $\{f + g + \ell = 0 \}$. The same holds for the field restricted to a smooth hypersurface. 
Hence, almost surely, on event $\{f + g + \ell \in \textrm{Cross}_k(R)\}$ or $\{f + g +  \ell \in A_1(r, R)\}$, there exist a witness path $\gamma$ and a $\delta>0$ such that $f+g +\ell>\delta$ on the image of $\gamma$. Hence the event is still verified by the field $f+ g + \ell\pm\delta'$ for all $\delta'\leq\delta$, which implies continuity.

\subsection{Proof of Lemma \ref{l:algo}}

We begin by introducing some notation. We make a distinction between the set of boxes that are \textit{revealed} by an algorithm, and the set $V \subset \R^d$ on which the field $f$ is \textit{determined} by an algorithm.
More precisely, for $V \subset \R^d$ and a set of boxes $\mathcal{P} \subset \mathcal{S}_s$, we say that $f$ is \textit{determined on $V$ using $\mathcal{P}$} if $f|_V = (\sum_{S \in \mathcal{P}} f_S)|_V$, or equivalently, if $((\bigcup_{S \in \mathcal{S}_s \setminus \mathcal{P}}\Supp(q\star \one_S))\cap V=\emptyset$.

Distinct boxes $S, S' \in \mathcal{S}_r$ are \textit{adjacent} if their closures have non-empty intersection. For a set of boxes $\mathcal{P} \subset \mathcal{S}_r$ define its \textit{outer boundary} 
\[ \partial^+ P :=  \{S \in \mathcal{S}_r \setminus  \mathcal{P} : \text{ $S$ is adjacent to a box $S' \in \mathcal{P}$ } \} , \]
so in particular $\partial^+ \{S\}$ are the boxes adjacent to $S$. Define also the \textit{interior} $\textrm{int}(\mathcal{P}) := \{ S \in  \mathcal{P} : \partial^+ \{S\} \subseteq \mathcal{P}  \}$. Note that, since $q$ is supported on $\Lambda_r$, $f$ is determined on $\textrm{int}(\mathcal{P})$ using $\mathcal{P}$. A \textit{primal (resp.\ dual) path} will designate a path in $\{f \geq 0\}$ (resp.\ $\{f \leq 0\}$). An \textit{interface} inside a compact set $K$ will designate, in the case $\X=\R$, a connected component of $\{f = 0\}\cap K$. In the case $\X=\Z$, an \textit{interface} will be a self-avoiding edge path $(e_1,...,e_k)$ inside $K$ such that
\begin{itemize}
    \item for any $i=1,...,k$, the endpoints of $e_i$ are in $\{f \ge 0\}$,
    \item for any $i$, there exists a vertex $v\in\Z^d$ such that $e_i$ and $v$ are at distance smaller or equal to $1$ for the sup norm, and $f(v)\le 0$, and
    \item the path is maximal for these properties.
\end{itemize}
In both settings, we say that these paths are \textit{contained in} a set of boxes $\mathcal{P} \subset \mathcal{S}_r$ if they are contained in $\cup_{S \in \mathcal{P}} S$. The \textit{left} and \textit{right} sides of $B_k(R)$ are respectively $\{-R\} \times [-kR, kR]^{d-1}$ and $\{R\} \times [-kR, kR]^{d-1}$, and if $d=2$ the \textit{top} and \textit{bottom} sides are defined similarly.

\smallskip
For the first statement consider the following algorithm:
\begin{itemize}
\item Reveal every box that intersects $\Lambda_1$ as well as all adjacent boxes. 
\item Iterate the following steps:
\begin{itemize}
\item  Let $\Wc \subset \mathcal{S}_r$ be the boxes that have been revealed. 
\item Identify the set $\Uc \subseteq \partial^+ (\textrm{int}(\Wc))$ such that, for each $S \in \Uc$, there is a primal path contained in $\textrm{int}(\Wc) \cap \Lambda_R$ between $\Lambda_1$ and the boundary of $S$ (measurable since $f$ is determined on $\textrm{int}(\Wc)$). In other words, $\Uc$ contains all boxes on which $f$ is not yet determined but which are connected to $\Lambda_1$ by a primal path in $\Lambda_R$ that \textit{has} been determined.
\item If $\Uc$ is empty end the loop. Otherwise reveal the boxes in $\partial^+ \Uc \setminus \Wc$.
\end{itemize}
\item If $\textrm{int}(\Wc) \cap \Lambda_R$ contains a primal path between $\Lambda_1$ and $\partial \Lambda_R$ output $1$, otherwise output~$0$.
\end{itemize}
This algorithm determines $A_1(R)$ since $\textrm{int}(\Wc)$ eventually contains all the components of $\{f \ge 0\} \cap \Lambda_R$ that intersect $\Lambda_1$. To estimate the sum of revealments $\rev(S)$, a box $S$ is revealed if and only if either (i) it is adjacent to a box that intersects $\Lambda_1$, or (ii) there is a primal path in $\Lambda_R$ between $\Lambda_1$ and a box adjacent to $S$. If $S = v + [0, r)^2$ and $\Lambda_1 \cap (v + \Lambda_{2r}) = \emptyset$ then the latter implies the occurrence of $\Lambda_1 \longleftrightarrow v + \Lambda_{2r}$. Summing over $S$ gives
\[ \sum_{S \in \mathcal{S}_r} \textrm{Rev}(S) \le      \sum_{v \in r \Z^d \cap \Lambda_{R+2r}} \P_p[ \Lambda_1 \longleftrightarrow v + \Lambda_{2r}] .\]

\smallskip
The second algorithm is the following
\begin{itemize}
\item Define $L$ to be $\{-R\}\times[-kR,kR]^{d-1}$, and reveal every box that intersects $L \cap B_k(R)$, as well as all adjacent boxes. 
\item Iterate the following steps:
\begin{itemize}
\item  Let $\Wc \subset \mathcal{S}_r$ be the boxes that have been revealed. 
\item Identify the set $\Uc \subseteq \partial^+ (\textrm{int}(\Wc))$ such that, for each $S \in \Uc$, there is a primal path contained in $\textrm{int}(\Wc) \cap B_k(R)$ between $L \cap B_k(R)$ and the boundary of $S$.
\item If $\Uc$ is empty end the loop. Otherwise reveal the boxes in $\partial^+ \Uc \setminus \Wc$.
\end{itemize}
\item If $\textrm{int}(\Wc) \cap B_k(R)$ contains a primal path between the left and right sides of $B_k(R)$ output $1$, otherwise output $0$.
\end{itemize}
  A box $S \in \mathcal{S}_r$ such that $d_\infty(S, B_k^+(R)) \le r$ is only revealed if there is a primal path in $B_k(R)$ between $L$ and a box adjacent to $S$, which implies the occurrence of (a translation of) the event $A_1(2r, d)$, where $d$ is the distance from the centre of $S$ to $L$. Since $d$ is at least $R-2r$, $\rev(S) \le \P_\ell[A_1(2r, R-2r)]$ as required.

\smallskip

The final algorithm (specific to $d=2$) is:
\begin{itemize}
\item Define $L_1=[-R,R]\times\{-kR\}$ and $L_2=\{-R\}\times[-kR,kR]$, and reveal every box that intersects $(L_1\cup L_2) \cap B_k(R)$, as well as all adjacent boxes.
\item Iterate the following steps:
\begin{itemize}
\item Let $\Wc \subset \mathcal{S}_r$ be the boxes that have been revealed. 
\item Identify the set $\Uc \subseteq \partial^+(\textrm{int}(\Wc))$ such that, for each $S \in \Uc$, there is an interface contained in $\textrm{int}(\Wc) \cap B_k(R)$ between $(L_1\cup L_2) \cap B_k(R)$ and the boundary of $S$.
\item If $\Uc$ is empty end the loop. Otherwise reveal the boxes in $\partial^+ \Uc \setminus \Wc$. 
\end{itemize}
\item If $\textrm{int}(\Wc) \cap B_k(R)$ contains a primal (resp.\ dual) path between the left and right (resp.\ top and bottom) sides of $B_k(R)$ terminate with output $1$ (resp.\ $0$).
\item Since $\textrm{int}(\Wc)$ contains interfaces inside $B_k(R)$ that intersect $L_1 \cap L_2$ and the algorithm has not yet terminated, exactly one of $\{f \ge 0\} \cap B_k(R)$ or $\{f \le 0\} \cap B_k(R)$ has an interface that intersects all four sides of $B_k(R)$. Partition $B_k(R)$ into regions $(P_i)$ using the interfaces inside $\cap B_k(R)$ that intersect $L_1\cup L_2$. Let $\mathbb{A}$ to be the region $P_i$ which contains the top-left corner of $B_k(R)$, and set $\mathcal{C}=1$ (resp.\ $\mathcal{C}=0$) if $f$ is positive (resp.\ negative) on $P_i$. Then iterate the following:
\begin{itemize}
\item If $\mathbb{A}$ contains a path in $B_k(R)$ between its left and right sides  terminate with output~$\mathcal{C}$.
\item Change the value of $\mathcal{C}$ (from $0$ to $1$ or $1$ to $0$), and add to $\mathbb{A}$ the region $P_i$ that is adjacent to it.
\end{itemize}
\end{itemize}
The final loop is illustrated in Figure \ref{f:alg}; it terminates almost surely since there are a finite number of interfaces in $B_k(R)$ (recall that when $\X=\R$, $f$ is $C^1$-smooth). Note that the algorithm does not necessarily reveal all components of $\{f = 0\}$ inside $B_k(R)$ -- any components which are closed loops or only touch the top and right sides of $B_k(R)$ are not revealed -- but these do not affect whether $\textrm{Cross}_k(R)$ occurs. 

\begin{figure}
\centering
\includegraphics[height=5cm]{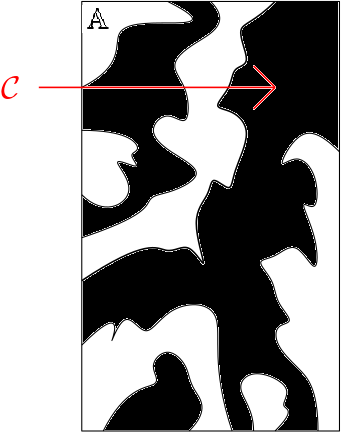}
\includegraphics[height=5cm]{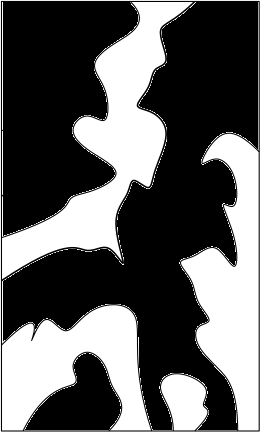}
\includegraphics[height=5cm]{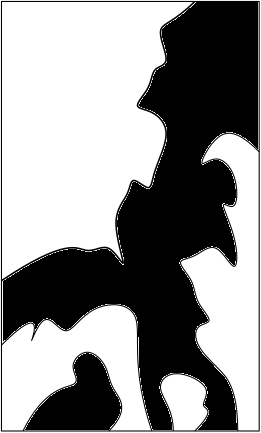}
\includegraphics[height=5cm]{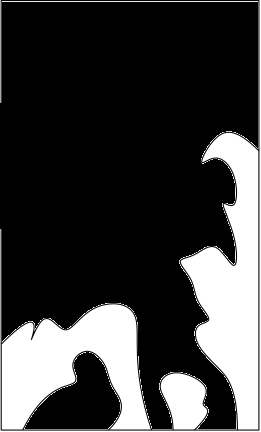}
\caption{The final loop of the algorithm in the proof of the third statement of Lemma \ref{l:algo}; this loop occurs when there is no left-right or top-bottom paths in $\{f=0\} \cap B_k(R)$. In this example the loop expands the area $\mathbb{A}$ three times in order to determine the sign $\mathcal{C}$ of the crossing.}
\label{f:alg}
\end{figure}

To estimate the revealments of this algorithm, a box $S \in \mathcal{S}_r$ such that $d_\infty(S, B_k^\dagger(R)) \le r$ is only revealed if there is an interface in $B_k(R)$ between $L_1 \cap L_2$ and a box adjacent to $S$, which implies the occurrence of (a translation of) the event $A_2(2r, d')$, where $d'$ is the distance from the centre of $S$ to $L_1 \cap L_2$. Since $d'$ is at least $R-2r$, $\rev(S) \le \P_\ell[A_2(2r, R-2r)]$ as required.

%%%%%%%%%%%
%%%%%%%%%%%
\medskip
\section{Extension to models with unbounded range of dependence}
\label{s:gf}

In this section we use approximation arguments to extend our results to models with unbounded range of dependence. We shall assume that $f$ satisfies Assumption \ref{a:gf} with parameter $\beta>d$, and also that \eqref{a:pos1} holds.

\smallskip
As well as the intermediate propositions \ref{p:ubgf1}--\ref{p:lbgf}, in this section we shall additionally rely on a slight variant of Proposition \ref{p:lbgf} which applies outside the self-dual case:

\begin{proposition}
\label{p:lbgf2}
Suppose $r>0$ is such that $q$ is supported on $\Lambda_r$. Then for $k \ge 1$ there exists $c = c(k)>0$ such that, for $\ell \in \R$ and $R \ge 8r > 0$,
\[ \frac{d^-}{d \ell} \mathbb{P}_\ell[ \textrm{Cross}_k(R)]  \geq  \frac{c}{ \|q\|_2 }\frac{ \P_\ell[\textrm{Cross}_k(R)]\big(1- \P_\ell[\textrm{Cross}_k(R)]\big) }{\frac{r}{R}  \sum_{i=2}^{R/r}\P_\ell[A_1(2r,ir)] }  . \]
\end{proposition}

\begin{proof}[Proof of Proposition \ref{p:lbgf2}]
The proof is very similar to that of Proposition \ref{p:lbgf}. Similarly to in Lemma \ref{l:algo}, for every $k \ge 1$, $\ell \in \R$, and $ R \ge 4r > 0$, there is an algorithm in $\mathcal{A}_r$ determining $\textrm{Cross}_k(R)$ such that, under $\mathbb{P}_\ell$, 
\[   \max_{S \in \mathcal{S}_r } \textrm{Rev}(S) \le \frac{4r}{R}\sum\limits_{i=2}^{R/r}\P_\ell[A_1(2r ,ir)]  .\]
Indeed we can define an algorithm which initially reveals a line $L$ of the form $L=\{i r\}\times[-kR,kR]^{d-1}$ for some uniform random $i$ in the given range. Then the algorithm reveals sequentially the set which is connected to that line through primal paths until the occurrence of the crossing event is determined. Then we deduce the result by applying Proposition \ref{p:lbderiv} (in the case $s=r$, $\mathcal{S}' = \mathcal{S}_s$, and $A = \textrm{Cross}_k(R)$) to this algorithm. 
\end{proof}

We shall also need an approximation lemma that allows us to compare a general $f = q \star W$ to a field with finite-range dependence. Fix a smooth symmetric cutoff function $\varphi: \R \to [0,1]$ such that $\varphi(x) = 1$ for $\|x\|_\infty \le 1/2$, $\varphi(x) = 0$ for $\|x\|_\infty \ge 1$. For $r > 0$ define 
\begin{equation}
\label{e:fr}
 f_r := q_r \star W 
 \end{equation}
where $q_r(x) := q(x) \varphi(|x|/r)$. Note that $q_r$ is supported on $\Lambda_r$, and also, since $q \in L^1(\X^d) \cap L^2(\X^d)$, as $r \to \infty$,
\begin{equation}
\label{e:qconv}
\|q_r\|_2 \to \|q\|_2  \qquad \text{and} \qquad \int q_r \to \int q. 
\end{equation}
Using the fact that Assumption \ref{a:gf} and \eqref{a:pos1} hold for $f$, we can verify that these hold also for $f_r$ for sufficiently large $r$ (on the other hand, deducing that \eqref{a:pos2} holds for $f_r$ if it holds for $f$ seems difficult but we do not need it). In particular, in the self-dual case $\ell_c(f_r) = \ell_c(f) = 0$.

\smallskip
The following lemma, essentially taken from \cite{mv20}, allows us to compare $f$ and $f_r$. We give details on the proof at the end of the section.

\begin{lemma}
\label{trunc}
There exist $c_1, c_2 > 0$ such that, for $r, R \ge 2$, increasing event $A$ measurable with respect to $f|_{B(R)}$, and $\ell \in \R$, 
\[ \left|\P_\ell[f\in A]-\P_\ell[f_r\in A]\right|\leq c_1 \big( R^{d/2} (\log R)  r^{-(\beta-d/2)} +  e^{ - c_2 (\log R)^2 } \big). \]
The same conclusion holds if $A$ is the intersection of one increasing and one decreasing event which are both measurable with respect to $f|_{B(R)}$.
\end{lemma}

We are now ready to prove Theorem \ref{t:oagf}:

\begin{proof}[Proof of Theorem \ref{t:oagf}]
In the proof $c > 0$ are constants that depend only on $f$ (and the choice of the cutoff function $\varphi$ in \eqref{e:fr}) and may change from line to line. 

\smallskip
The bound $\eta_1 \le d-1$, and also $\eta_1 \le 1/2$ if $d=2$ and \eqref{a:pos2} holds, are rather classical; in fact they are true for any $\beta > d$. For the former, combining $\mathbb{P}_{\ell_c}[\textrm{Cross}_5(R)] \ge \delta$ (the first statement of Lemma \ref{l:classicalgf}) with the union bound applied along the hyperplane $\{0\} \times [-kR, kR]$ gives $\P_{\ell_c}[A_1(R)] \ge c R^{-(d-1)}$. For the latter, by combining the RSW estimates (the second statement of Lemma~\ref{l:classicalgf}) with Lemma~\ref{trunc} one can deduce (see \cite{bg17, mv20} for similar arguments)
\[ \P_{\ell_c}\big[   \{-R\} \times [-R,R]   \stackrel{B_1(R)}{\doublepath} \{R\} \times [-R,R]  \big]  \ge c \big( 1-   R^{1 -(\beta-1)}(\log R) \big) \ge c / 2 \]
for sufficiently large $R$. By the union bound applied along $\{0\} \times [-R, R]$ this implies $\P_{\ell_c}[A_2(R)] \ge c R^{-1}$, and given \eqref{e:eta1eta2} we see that  $\P_{\ell_c}[A_1(R)] \ge c R^{-1/2}$. 

\smallskip
We now prove the remaining bounds, beginning with the first statement. Fix $1 > \alpha > \frac{d/2}{\beta-d/2}$ and $\eta_1 > \eta^\ast > 0$ (if $\eta_1 = 0$ there is nothing to prove). Then by monotonicity in $\ell$, the union bound, and the definition of $\eta_1$,
\[ \mathbb{P}_\ell[ A_1(r, R ) ]  \le  \mathbb{P}_{\ell_c}[ A_1(r,R) ]  \le c r^{d-1} \mathbb{P}_{\ell_c}[A_1(R-r)]  \le r^{d-1} R^{-\eta^\ast}   \]
for all $\ell \le \ell_c$, $R$ sufficiently large, and $r \in [1, R/2]$. Set $r = R^\alpha$. Then by an integral comparison,
  \begin{align*}
  \frac{r}{R} \sum_{i=2}^{R/r} \mathbb{P}_{\ell_c}[ A_1(2r,ir )] & \le  c  r^{-\eta^\ast+(d-1)}  \times  \frac{r}{R} \sum_{i=2}^{R/r} i^{-\eta^\ast}   \\
  &  \le  c  r^{-\eta^\ast+(d-1)} (R/r)^{-\min\{\eta^\ast,1\}}(\log (R/r) ) \\
  & \le c (\log R) \big( R^{\alpha(d-1-\eta^\ast) -(1-\alpha)\min\{\eta^\ast,1\} } \big) 
  \end{align*}
  for $\ell \le \ell_c$ and large $R$. Consider the field $f_r$ defined in \eqref{e:fr}. By Lemma~\ref{trunc}, 
\[ \P_\ell[f_r\in A_1(r', R)]  \le \P_\ell[A_1(r', R)]  + c R^{d/2  - \alpha(\beta-d/2)} (\log R) + c e^{ - c (\log R)^2 }  \]
 for $\ell \le \ell_c$ and $2 \le r' \le R$, and hence 
 \[ \frac{r}{R} \sum_{i=2}^{R/r} \mathbb{P}_{\ell}[f_r \in A_1(2r,ir )]  \le c (\log R) \big( R^{\alpha(d-1-\eta^\ast) -(1-\alpha)\min\{\eta^\ast,1\} }  + R^{d/2 - \alpha(\beta-d/2)} \big) \]
 for $\ell \le \ell_c$ and large $R$. Moreover, by Lemma \ref{l:classicalgf} there are $\delta > 0$ and $\ell' = \ell'(R) \le \ell_c$ such that $ \P_{\ell'}[\textrm{Cross}_5(R)] = \delta$. Hence, again by Lemma \ref{trunc},
\[ \P_{\ell'}[f_r \in \textrm{Cross}_5(R)]\big(1-\P_{\ell'}[f_r \in \textrm{Cross}_5(R)]\big)  \ge \delta(1-\delta) - c R^{d/2 - \alpha(\beta-d/2)} (\log R)  \ge \delta(1-\delta)/2   \]
for large $R$, where we used that $\alpha > \frac{d/2}{\beta-d/2}$.

\smallskip
We now apply Propositions \ref{p:ubgf2} and \ref{p:lbgf2} to the field $f_r$ at the sequence of levels $\ell'(R) \le \ell_c$. First, by Proposition \ref{p:ubgf2} (recalling \eqref{e:qconv})
\begin{align}
\label{e:oab1gf}
 \frac{d}{d \ell} \mathbb{P}_\ell[ f_r \in \textrm{Cross}_5(R)] \Big|_{\ell= \ell'}  & \ge \frac{c \delta(1-\delta)}{  \|q_r\|_2}  \Big( \frac{r}{R} \sum_{i=2}^{R/r} \mathbb{P}_{\ell'}[f_r \in A_1(2r,ir )] \Big)^{-1}  \\
 \nonumber & \ge  c (\log R)^{-1}  \big( R^{\alpha(d-1-\eta^\ast) -(1-\alpha)\min\{\eta^\ast,1\} }  +  R^{d/2 - \alpha(\beta-d/2)}  \big) ^{-1} 
 \end{align}
for large $R$. Similarly, by Proposition \ref{p:ubgf1},
\begin{align}
\label{e:oab2gf}
 \frac{d}{d \ell} \mathbb{P}_\ell[ f_r \in \textrm{Cross}_5(R) ] \Big|_{\ell= \ell'} & \le  c R^{d/2 }  \Big(  \mathbb{P}_{l'}[f_r \in A_1(2r,R)] \Big)^{1/2}  \\
  \nonumber & \le  c R^{d/2} \big( R^{- \eta^\ast + \alpha(d-1) }  +  R^{d/2 - \alpha(\beta-d/2) }(\log R) \big)^{1/2} \\
 \nonumber & \le  c \sqrt{\log R} \big( R^{d/2- \eta^\ast/2 + \alpha(d-1)/2 }  +  R^{3d/4 - \alpha(\beta-d/2)/2} \big)
 \end{align}
for large $R$, where we used that $\sqrt{a+b} \le \sqrt{a} + \sqrt{b}$ for $a,b > 0$. Comparing \eqref{e:oab1gf} and \eqref{e:oab2gf} and expanding the brackets we deduce that at least one of the exponents
 \begin{align*}
    & E_1 :=\big(3d/4 - \alpha(\beta-d/2)/2 \big) + \big(d/2 - \alpha(\beta-d/2) \big)  \\
    & E_2 := \big(d/2- \eta^\ast/2 + \alpha(d-1)/2\big) + \big(d/2 - \alpha(\beta-d/2)\big)    \\
   & E_3 := \big(d/2- \eta^\ast/2 + \alpha(d-1)/2 \big)  + \big( \alpha(d-1-\eta^\ast) - (1-\alpha)\min\{\eta^\ast, 1\} \big) \\
   &  E_4 :=\big(3d/4 - \alpha(\beta-d/2)/2 \big) +  \big( \alpha(d-1-\eta^\ast) - (1-\alpha)\min\{\eta^\ast, 1\} \big) 
   \end{align*}
 must be non-negative. The first is equivalent to $\alpha \le \frac{5d}{6(\beta-d/2)}$. The second implies that $\eta^\ast \le \frac{d}{3} + \alpha(d-1)$, assuming that $\alpha > \frac{5d}{6(\beta-d/2)}$. The third is equivalent to either $\eta^\ast \le \frac{d}{3}  + \alpha(d-1)$ (if $\eta^\ast \le 1$) or $\eta^\ast \le \frac{d-2+\alpha(3d-1)}{1+2\alpha}$ (if $\eta^\ast > 1$). Finally, the fourth implies either $\eta^\ast \le \frac{d}{3}  + \alpha(d-1)$ (if $\eta^\ast \le 1$, assuming that $\alpha > \frac{5d}{6(\beta-d/2)}$) or $\alpha \le \frac{3d-4}{2\beta-5d + 4}$ (if $\eta^\ast > 1$). One can check that, since $d\ge3$, $ \frac{5d}{6(\beta-d/2)} < \frac{3d-4}{2\beta-5d + 4}$ and $ \frac{d}{3}  + \alpha(d-1) <  \frac{d-2+\alpha(3d-1)}{1+2\alpha}$. Hence we conclude that if $ \alpha > \frac{3d-4}{2\beta-5d + 4}$ then $\eta^\ast \le  \frac{d-2+\alpha(3d-1)}{1+2\alpha}$. Sending $\alpha \to  \frac{3d-4}{2\beta-5d + 4}$ from above gives the result.

\smallskip
The proof of the remaining statements are similar, and closer to the arguments in Section~\ref{s:bou}. For the second statement, fix $1 > \alpha > \frac{3d/2-1}{\beta-d/2}$ and $\eta_1 > \eta^\ast > 0$. As in the proof of the first statement, 
\begin{equation}
\label{e:oab3gf}
 \mathbb{P}_{\ell_c}[ A_1(2r, R ) ]  \le r^{d-1} R^{-\eta^\ast} 
 \end{equation}
for large $R$ and $r \in [1, R/4]$. Now let $r = R^\alpha$. Since we have the a priori bound $\P_{\ell_c}[A_1(1R) ] \ge c R^{-(d-1)}$ (from the start of the proof), by Lemma~\ref{trunc}
\begin{align}
\label{e:compare}
  | \P_{\ell_c}[f_r \in A_1(r', R')]  -  \P_{\ell_c}[A_1(r', R') ] | & \le   c R^{d/2  - \alpha(\beta-d/2)} (\log R) + c e^{ - c (\log R)^2 } \\
  \nonumber & \le   \P_{\ell_c}[A_1(r', R')]  / 2
 \end{align}
 for large $R$ and $1 \le r' \le R' \le R$, where we used that $d/2 - \alpha(\beta-d/2) \le -(d-1)$ by the definition of $\alpha$. 
 
  Similarly to \eqref{e:compare} we also have
\begin{align*}
  |  \P_{\ell'}[f_r \in A_1(R)]   -  \P_{\ell'}[ A_1(R)]   | &  \le   c R^{d/2  - \alpha(\beta-d/2)} (\log R) + c e^{ - c (\log R)^2 } \\
  &  \le   \P_{\ell_c}[ A_1(R)]   =  \P_{\ell'}[ A_1(R)]/2   . 
  \end{align*}
Then applying Proposition \ref{p:ubgf1} to the field $f_r$, for large $R$,
\begin{align*} 
 \mathbb{P}_{\ell_c}[ A_1(R)] & = \mathbb{P}_{\ell'}[ A_1(R)] - \mathbb{P}_{\ell_c}[ A_1(R)]    \le 2 ( \mathbb{P}_{\ell'}[ f_r \in A_1(R)] - \mathbb{P}_{\ell_c}[  f_r \in A_1(R)]   )  \\
&  \le \frac{2 r^{d/2}(\ell'-\ell_c)}{\int q}  \sqrt{  \mathbb{P}_{\ell'}[A_1(R)]   \! \! \sum_{v \in r \Z^d \cap \Lambda_{R+2r}} \! \! \P_{\ell_c}[ \Lambda_{2r} \longleftrightarrow v + \Lambda_{2r} ]}   \\
 &  \le  \frac{2(\ell'-\ell_c)}{\int q}   R^{\alpha d/2}  \sqrt{ R^{-\eta^\ast}   R^{\max\{ 0, \alpha(d-1), -\alpha+d-2\eta^\ast \} } (\log R) } .
 \end{align*}
Comparing with \eqref{e:mfbuse} implies that  $\alpha d -\eta^\ast + \max\{ \alpha(d-1), -\alpha+d-2\eta^\ast \} \ge 0$, and so $\eta^\ast \le \max\{ d/3 +  \alpha(d-1)/3, \alpha (2d-1) \}$, and sending $\alpha \to  \frac{3d/2-1}{\beta-d/2}$ from above gives the result.

\smallskip
Finally, consider the third statement. Fix  $1 > \alpha > \frac{5}{3(\beta-1)}$ and $r = R^\alpha$. By the RSW estimates (the second statement of Lemma \ref{l:classicalgf}) and Lemma \ref{trunc},
\[    \P_{\ell_c}[f_r \in \textrm{Cross}_1(R)]\big(1-\P_{\ell_c}[f_r \in \textrm{Cross}_1(R)] \big) \le c - c  R^{1 -\alpha(\beta-1)} (\log R) < c /2  \]
for large $R$. Then by Propositions \ref{p:ubgf2} and \ref{p:lbgf} we have, for large $R$,
\[  c  \mathbb{P}_{\ell_c}[f_r \in A_2(2r,R-2r )]^{-1}   \le  \frac{d}{d \ell} \mathbb{P}_\ell[ f_r \in \textrm{Cross}_1(R)] \Big|_{\ell= \ell_c}    \le   c  R \sqrt{\mathbb{P}_{\ell_c}[f_r \in A_2(2r,R-2r ) ]}  \]
which gives $ \mathbb{P}_{\ell_c}[f_r \in A_2(2r, R-2r )]   \ge  c R^{-2/3}$ for large $R$. Applying the union bound and  Lemma \ref{trunc} (valid since $A_2(3 \sqrt{2}r, R)$ is the intersection of an increasing and a decreasing event) yields
\[  \mathbb{P}_{\ell_c}[A_2(R-2r )]  \ge c r^{-1} \mathbb{P}_{\ell_c}[A_2(2r,R-2r )] \ge c  R^{-\alpha} \big( R^{-2/3} -  R^{1 - \alpha(\beta-1)} (\log R)  \big) . \]
Sending $\alpha \to \frac{5}{3(\beta-1)}$ from above gives that, for every $\varepsilon > 0$, 
 \begin{equation}
\label{e:twoarm2}
 \mathbb{P}_{\ell_c}[A_2(R)] \ge c_2   R^{-2/3 - 5/(3(\beta - 1)) - \varepsilon}   . 
 \end{equation}
 for $c_2 = c_2(\varepsilon) > 0$ and large $R$. Hence by \eqref{e:eta1eta2} (recall that here $\X=\R$),
 \begin{equation}
\label{e:limsup4}
\mathbb{P}_{\ell_c}[A_1(R )]  \ge ( \mathbb{P}_{\ell_c}[A_2(R )] )^{1/2} \ge  c_3 R^{-1/3- 5/(6(\beta-1)) - \varepsilon/2} 
\end{equation}
for $c_3 = c_3(\varepsilon) > 0$ and large $R$, which gives the result. 
\end{proof}

\subsection{The approximation lemma}
To finish the section we give some details on the approximation result in Lemma~\ref{trunc}.

\begin{proof}[Proof of Lemma \ref{trunc}]
We first observe that $g := f - f_r = (q-q_r) \star W$ is a stationary Gaussian field satisfying 
\[ \E[ g(0)^2 ] = \int_{x \in \R^d} (q-q_r)^2(x) \, dx  =  \int_{|x| > r/2} q(x)^2 (1 - \varphi(|x/r|))^2 \, dx \le  \int_{|x| > r/2} q(x)^2  \le c_1  r^{d-2\beta} ,   \] 
for some $c_1 > 0$ and we used that $|q(x)| \le  c |x|^{-\beta}$ by Assumption \ref{a:gf}. Similarly, in the case that $f$ is continuous, $g$ is $C^1$-smooth and for every direction $v \in \mathbb{S}^1$,
\[ \E[ (\partial_v g(0) )^2 ]  =  \int_{ |x| > r/2} (\partial_v (q(x)(1 - \varphi(|x/r|)) )^2 \, dx   \le c_2  r^{d-2\beta} ,   \] 
for some $c_2 > 0$ that depends on the (uniformly bounded) derivatives of $\varphi$, and we used that $|\nabla q(x)| \le  c |x|^{-\beta}$ by Assumption \ref{a:gf}. Then by a Borell-TIS argument (see \cite[Proposition 3.11]{mv20} for the case $f$ is continuous and $d=2$, and the proof is similar in the general case) there exist $c_3, c_4 > 0$ such that, for all $R, r \ge 2$,
\begin{equation}
\label{e:btis}
 \P [  \| f - f_r \|_{\infty, B(R)} >  c_3 (\log R) r^{-(\beta - d/2)} ]   \le c_3 e^{-c_4 (\log R)^2}  
 \end{equation}
We also note the following consequence of \eqref{a:pos1} which can be proved with a Cameron-Martin argument (see \cite[Proposition 3.6]{mv20} for the case $d=2$, and the proof is identical in all dimensions): there exists a $c_5 > 0$ such that, for $R \ge 1$, increasing event $A'$ measurable with respect to $f|_{B(R)}$, $\ell \in \R$ and $t > 0$,
\begin{equation}
\label{e:cm}
  \P_\ell[ \{f  + t \in A' \} \setminus \{ f  \in A' \} ] =  \P_\ell[f  + t  \in A' ]  - \P_\ell[ f  \in A'  ] \le c_5 t R^{d/2} .
  \end{equation}

We now complete the proof, for which we may assume that $\ell = 0$. Consider $A = A_1 \cap A_2$ where $A_1$ is increasing, $A_2$ is decreasing, and both $A_1$ and $A_2$ are measurable with respect to $f|_{B(R)}$. Abbreviate $t = c_3 (\log R) r^{-(\beta-d/2)}$ and define $E = \{ \| f - f_r \|_{\infty, B(R)} > t \}$. Then 
\begin{align*}
 & \P[ f_r \in A_1 \cap A_2 ]  \le \P[ f_r \in A_1 \cap A_2 \cap E^c]  + \P[E]  \\
 & \quad \le \P[ \{ f + t \in A_1 \} \cap \{f - t \in A_2 \}   ]  + \P[E]  \\
 & \quad \le \P[ f \in A_1 \cap A_2 ] + \P[ \{f + t \in A_1 \} \setminus \{ f   \in A_1 \} ] +  \P[ \{f -t \in A_2 \} \setminus \{ f  \in A_2 \} ]   + \P[E] \\
 & \quad \le  \P[ f \in A_1 \cap A_2 ]  +   2 c_5  t R^{d/2}   +  c_3 e^{-c_4 (\log R)^2}
 \end{align*}
 where in the second inequality we used that $A_1$ (resp.\ $A_2$) is increasing (resp.\ decreasing) and measurable with respect to $f|_{B(R)}$, and the final inequality was by \eqref{e:btis} and \eqref{e:cm}. Similarly
 \begin{align*}
 & \P[ f_r \in A_1 \cap A_2 ]    \ge \P[ \{ f - t \in A_1 \} \cap \{f + t \in A_2 \}  \cap E^c ]  \\
 & \quad \ge \P[ f \in A_1 \cap A_2] - \P[ \{f  \in A_1 \} \setminus \{ f -t  \in A_2 \} ] - \P[ \{f  \in A_2 \} \setminus \{ f + t  \in A_2 \} ]  - \P[E] \\
 & \quad \ge \P[ f \in A_1 \cap A_2] -  2 c_5  t R^{d/2}   - c_3 e^{-c_4 (\log R)^2}
 \end{align*}
which gives the result.
\end{proof}

%%%%%%%%%
\medskip
\section{The Russo-type inequality for Gaussian fields and applications}
\label{s:osss}

In this section we prove the Russo-type inequality in Proposition \ref{p:russo}, and deduce Theorem~\ref{t:expdecay} as an application. We emphasise that in this section \eqref{a:pos1}--\eqref{a:pos2} play no role.

The main idea in the proof of Proposition \ref{p:russo}, which distinguishes it from the approach in \cite{mv20} (see Proposition \ref{p:russodis}), is to use an orthonormal decomposition of each $f_S$ to interpret $\frac{d^-}{d \ell} \mathbb{P}_\ell[A]$ and the resampling influences $\Infl_{\ell;A}(S)$ as measuring, respectively, the `boundary' and `volume' of certain sets in Gaussian space. Then we apply Gaussian isoperimetry to deduce the result. For a set $E \subset\mathbb{R}^n$ we denote
\[ E^{+\varepsilon} := \{ x \in \mathbb{R}^n : \text{ there exists } y \in E \text{ s.t. } \|x-y\|_2 \le \varepsilon \} \]
to be the $\varepsilon$-thickening of $E$.

\begin{proposition}[Gaussian isoperimetry]
\label{p:iso}
There exists a constant $c > 0$ such that, for every measurable $E \subset\mathbb{R}^n$ and $\varepsilon \ge 0$,
\[ \mathbb{P}[X \in E^{+\varepsilon} \setminus E] \ge   \sqrt{\frac{2}{\pi}}    \mathbb{P}[X \in E] (1 - \mathbb{P}[X \in E]) \varepsilon    -  c \varepsilon^2  \]
where $X$ is an $n$-dimensional standard Gaussian vector.
\end{proposition}
\begin{proof}
Let $\varphi(x)$ and $\Phi(x)$ denote the standard normal pdf and cdf respectively. The classical Gaussian isoperimetric inequality states that
\[  \liminf_{\varepsilon \downarrow 0} \varepsilon^{-1}  \mathbb{P}[X \in E^{+\varepsilon} \setminus E]     \ge \varphi( \Phi^{-1}( \mathbb{P}[X \in E] )   ) .  \]
A simple consequence (see, e.g., \cite[Eq. (3)]{led98}) is that, for any $\varepsilon \ge 0$,
\begin{equation}
\label{e:iso1}
 \mathbb{P}[X \in E^{+\varepsilon}] \ge   \Phi(\Phi^{-1}(\mathbb{P}[X \in E]) + \varepsilon  ). 
 \end{equation}
 By Taylor expanding $\Phi$ on the right-hand side of \eqref{e:iso1} we have 
 \[ \mathbb{P}[X \in E^{+\varepsilon} \setminus E] \ge   \varepsilon \varphi(\Phi^{-1}(\mathbb{P}[X \in E]) ) - \frac{1}{2} \sup_{x \in \mathbb{R}} |\varphi'(x)| \varepsilon^2    , \]
 and the result follows since, for all $x \in \mathbb{R}$, $\varphi(x) \ge \sqrt{\frac{2}{\pi}} \, \Phi(x) (1-\Phi(x))$(as can be seen from the fact that the Mill's ratio $(1-\Phi(x))/\varphi(x)$ is decreasing on $x \ge 0$), and since $|\varphi'(x)|$ is bounded over $\mathbb{R}$.
\end{proof}

We use the following orthogonal decomposition of $f_S$ (see Proposition \ref{p:odecom} in the appendix). Let $Z = (Z_i)_{i \ge 1}$ be a sequence of i.i.d.\ standard normal random variables and let $(\varphi_i)_{i \ge 1}$ be an orthonormal basis of $L^2(S)$. Then 
\begin{equation}
\label{e:odecom}
 f^n_S := \sum_{i \ge 1}^n Z_i (q \star \varphi_i) \Rightarrow f_S , 
 \end{equation}
in law with respect to the $C^0(\X^d)$-topology.

\begin{proof}[Proof of Proposition \ref{p:russo}]
By linear rescaling, we may suppose without loss of generality that $\ell = 0$ and $\|q\|_2 = 1$. For each $S \in \mathcal{S}_s$, let $g_S : \R^d \to [0,1]$ be a smooth function such that $g_S(x) = 1$ on $\{x : d_\infty(x, S) \le r + s\}$ and $g_S(x) = 0$ on $\{x : d_\infty(x, S) \ge 2(r+s)\}$. Then $\sum_{S \in \mathcal{S}_s}g_S(x) \le c_1 \max\{1, (r/s)^d\}$ for some constant $c_1 > 0$ depending only on $d$. Therefore, since $A$ is increasing, by the multivariate chain rule for Dini derivatives,
\begin{equation}
\label{e:rusfin0}
 \frac{d^-}{d \varepsilon} \mathbb{P}[f + \varepsilon \in A]  \Big|_{\varepsilon = 0} \ge \frac{1}{c_1 \max\{1, (r/s)^d\}}  \sum_{S \in \mathcal{S}_s} \frac{d^-}{d \varepsilon} \mathbb{P}[f + \varepsilon g_S \in A]  \Big|_{\varepsilon = 0} . 
 \end{equation}
For each $S \in \mathcal{S}_s$, let $f'_S$ denote an independent copy of $f_S$, define $h_S = f - f_S$, and let $\mathcal{F}_{h_S}$ be the $\sigma$-algebra generated by $h_S$. We next claim that, almost surely over $\mathcal{F}_{h_S}$,
\begin{equation}
\label{e:rusfin}
\frac{d^-}{d \varepsilon} \mathbb{P}[f_S + h_S +  \varepsilon g_S \in A | \mathcal{F}_{h_S} ] \Big|_{\varepsilon = 0}   \ge     c_2   \mathbb{P}[ \id_{\{f_S + h_S  \in A\}} \neq \id_{\{f'_S + h_S \in A \}} | \mathcal{F}_{h_S} ]  
\end{equation}
for some universal $c_2 > 0$. Together with \eqref{e:rusfin0}, this will complete the proof of Proposition \ref{p:russo} since
\begin{align*}
 \frac{d^-}{d \varepsilon} \mathbb{P}[f_S + \varepsilon g_S \in A]  \Big|_{\varepsilon = 0} & \ge  \E \Big[ \frac{d^-}{d \varepsilon} \mathbb{P}[f_S + h_S +  \varepsilon g_S \in A | \mathcal{F}_{h_S} ] \Big|_{\varepsilon = 0}  \Big]  \\
 & \ge c_2 \E \big[      \mathbb{P}[ \id_{\{f_S + h_S  \in A\}} \neq \id_{\{f'_S + h_S \in A \}} | \mathcal{F}_{h_S}  ]  \big] \\
 & =:  c_2 \Infl_{0;A}(S) .
 \end{align*}
where the first inequality is Fatou's lemma, and the second inequality is by \eqref{e:rusfin}.

\smallskip
It remains to prove \eqref{e:rusfin}. Henceforth we fix $S \in \mathcal{S}_s$, condition on $h_S$, and drop $\mathcal{F}_{h_S}$ from the notation. Let $(\varphi_i)_{i \ge 1}$ be an orthonormal basis of $L^2(S)$ and recall the decomposition \eqref{e:odecom}. Fixing $n \in \mathbb{N}$ and viewing $\{f_S^n + h_S \in A\}$ as a Borel set $E$ in the $n$-dimensional Gaussian space generated by the standard Gaussian vector $Z^n = (Z_i)_{1 \le i \le n}$, by Proposition \ref{p:iso}
\begin{equation}
\label{e:rus2}
   \mathbb{P}[Z^n \in E^{+\varepsilon} \setminus E]   \ge   c_3  \varepsilon   \mathbb{P}[Z^n \in E] (1 - \mathbb{P}[Z^n \in E])  -  c_4 \varepsilon^2  
   \end{equation}
for some $c_3, c_4 > 0$ and every $\varepsilon \ge 0$. Consider $y  = (y_i) \in \mathbb{R}^n$ such that $\| y \|_2 = \varepsilon$. By the Cauchy-Schwarz inequality, and since $\varphi_i$ are an orthonormal basis,
\[  \Big\|   \sum_{i \le n} y_i  (q \star   \varphi_i)   \Big\|_\infty  =  \Big\|  q \star  \Big( \sum_{i \le n} y_i    \varphi_i \Big)   \Big\|_\infty \le \|q \|_2 \Big\| \sum_{i \le n} y_i \varphi_i \Big\|_2  =    \|y\|_2 = \varepsilon .  \]
Since $ q \star \varphi_i$ is supported on  $\{x : d_\infty(x, S) \le r+s\}$, and recalling that $g_S(\cdot) \ge \id_{d_\infty(\cdot, S) \le r+s}$, this gives
\[   \sup_{y : \|y\|_2 \le \varepsilon}  \, \sum_{i \le n} (Z_i +y_i )( q \star \varphi_i ) - f^n_S  =   \sup_{y : \|y\|_2 \le \varepsilon}  \,   \sum_{i \le n} y_i  (q \star   \varphi_i)   \le  \varepsilon  g_S . \]
Therefore, since $A$ is increasing,
\begin{align*}
    \mathbb{P}[f^n_S +  h_S + \varepsilon  g_S \in A]  - \mathbb{P}[ f^n_S + h_S \in A ]   &\ge    \mathbb{P}[\cup_{y : \|y\|_2 \le \varepsilon}  \{ Z^n + y \in E \} ] - \mathbb{P}[Z^n \in E]  \\
    &  =      \mathbb{P}[Z^n \in E^{+\varepsilon} \setminus E ]  .  
  \end{align*}
Combining with \eqref{e:rus2},
\begin{equation}
\label{e:rusfin2}
    \mathbb{P}[f^n_S +  h_S + \varepsilon g_S \in A]  - \mathbb{P}[ f^n_S + h_S \in A ]  \ge   c_3   \varepsilon   \mathbb{P}[f^n_S + h_S \in A] (1 - \mathbb{P}[f^n_S + h_S \in A])  - c_4 \varepsilon^2   .
    \end{equation}
It remains to prove that almost surely (with respect to $h_S$), as $n \to \infty$,
\begin{equation}
\label{e:cont1}   \mathbb{P}[f^n_S + h_S \in A] \to   \mathbb{P}[f_S + h_S \in A]  \quad \text{and} \quad  \mathbb{P}[f^n_S +  h_S + \varepsilon g_S \in A]  \to  \mathbb{P}[f_S +  h_S + \varepsilon g_S \in A]  , 
\end{equation} 
since then sending $n \to \infty$ in \eqref{e:rusfin2} yields
\[     \mathbb{P}[f_S +  h_S + \varepsilon g_S \in A]  - \mathbb{P}[ f_S + h_S \in A ]   \ge c_3 \varepsilon  \mathbb{P}[f_S + h_S \in A] (1 - \mathbb{P}[f_S + h_S \in A])  - c_4 \varepsilon^2, \]
 which gives \eqref{e:rusfin} after sending $\varepsilon \to 0$.
 
\smallskip 
So let us justify \eqref{e:cont1}. Recall that $A$ is an increasing continuity event; this means that for almost every $f = f_S + h_S$ there exists $\delta > 0$ such that  
\[   \id_{\{f_S + h_S  + s \in A\}}  \qquad \text{and} \qquad   \id_{\{ f_S + h_S + \varepsilon g_S + s \in A \}}  \]
are constant for $s \in (-\delta,\delta)$. Then since $f^n_S \to f_S$ in law with respect to the $C^0(\X^d)$-topology, we have \eqref{e:cont1} (by the Portmanteau lemma for instance).
\end{proof}

\begin{remark}
Note that in the proof of Proposition \ref{p:russo} we did not require that the Borel set $E$ in the $n$-dimensional Gaussian space generated by $Z^n$ be increasing, since Gaussian isoperimetry is valid for arbitrary sets. Hence we do not need $q \star \varphi_i$ to be a positive function, which allows us to lift the requirement that $q \ge 0$ appearing in \cite{mv20}.
\end{remark}

\begin{remark}
\label{r:refine}
The proof shows that the inequality could be strengthened by replacing $\frac{d^-}{d \ell} \mathbb{P}_\ell[A]   =   \frac{d^-}{d \varepsilon} \P_\ell[f   + \varepsilon \in A] $ with $\frac{d^-}{d \varepsilon} \P_\ell[f   + \varepsilon g_{S'} \in A]$ where $g_{\mathcal{S}'}(\cdot) := \id_{d_{\infty}(\cdot, \mathcal{S}' ) \le 2(r+s)} $.
\end{remark}

\subsection{The sharpness of the phase transition for finite-range models}
As an application we prove Theorem \ref{t:expdecay} following closely the approach in \cite{dcrt19a}. For this we only need the special case $s=r$ and $\mathcal{S}' = \mathcal{S}_s$ of Proposition \ref{p:lbderiv}.

\begin{proof}[Proof of Theorem \ref{t:expdecay}]

We prove the result with $\Lambda_{2r}$ in place of $\Lambda_1$, since the result for $\Lambda_1$ follows from the union bound.

\smallskip
For $R \ge 0$ define $g_R(\ell):= \mathbb{P}_\ell[A_1(2r, R)]$ (by convention $g_R(\ell) = 1$ if $R \in [0, 2r]$), and its limit $g_R :=\lim_{R \to \infty} g_R(\ell) = \P_\ell[ \Lambda_{2r} \longleftrightarrow \infty]$. We will first establish the differential inequality 
\begin{equation}
\label{e:dcrt}
 \frac{d^-}{d \ell} g_R(\ell)  \ge   \frac{c_1 g_R(\ell)(1-g_R(\ell))}{ \frac{1}{R} \sum_{i = 0}^{R-1}  g_i(\ell) } 
 \end{equation}
for some $c_1 > 0$, every $R$ sufficiently large, and every $\ell \in \R$. Recall the notation from the beginning of the proof of Lemma \ref{l:algo} and for $R \ge 2r$ consider the following algorithm in $\mathcal{A}_r$ (essentially taken from~\cite{dcrt19a}):
\begin{itemize}
\item Draw a random integer $i$ uniformly in $[2r, R]$, and reveal every box that intersects $\partial \Lambda_i$, as well as all adjacent boxes. 
\item Iterate the following steps:
\begin{itemize}
\item Let $\Wc \subset \mathcal{S}_1$ be the boxes that have been revealed. 
\item Identify the set $\Uc \subseteq \partial^+ (\textrm{int}(\Wc))$ such that, for each $S \in \Uc$, there is a primal path contained in $\textrm{int}(\Wc) \cap \Lambda_R$ between $\partial \Lambda_i$ and the boundary of $S$.
\item If $\Uc$ is empty end the loop. Otherwise reveal the boxes in $\partial^+ \Uc \setminus \Wc$.
\end{itemize}
\item If $\textrm{int}(\Wc)$ contains a primal path between $\Lambda_{2r}$ and $\Lambda_R$ output $1$, otherwise output $0$.
\end{itemize}
This algorithm determines $A_1(2r, R)$ since $\textrm{int}(\Wc)$ eventually contains all the components of $\{f \ge 0\} \cap \Lambda_R$ that intersect $\partial \Lambda_i$, and any primal path between $\Lambda_{2r}$ and $\Lambda_R$ must intersect $\partial \Lambda_i$. To estimate the revealments $\rev(S)$ under $\P_\ell$, note that a box $S$ is revealed if and only if either (i) it intersects, or is adjacent to a box that intersects, $\partial \Lambda_i$, or (ii) there is a primal path in $\Lambda_R$ between $\partial \Lambda_i$ and a box adjacent to $S$. If $d'$ denotes the distance from the centre of $S$ to $\Lambda_i$, this implies the occurrence of (a translation of) the event $A_1(2r, d')$. Averaging on $i \in [2r,R]$, we have 
\[ \rev(S) \le  \frac{1}{R-2r}  \Big(  4 + 2 \sum_{i = 2r}^{R-1} \P_\ell[A_1(2r,i)] \Big) \le  \frac{4}{R-1}  \sum_{i = 0}^{R-2r} g_i(\ell)    \le  \frac{c_2}{R}  \sum_{i = 0}^{R-1} g_i(\ell)  \]
for some $c_2 > 0$ and sufficiently large $R$.  Applying Proposition \ref{p:lbderiv} (with $s=r$ and $\mathcal{S}' = \mathcal{S}_1$, recalling that $A_1(2r, R)$ is a continuity event by Lemma \ref{l:classicalgf}) gives that
\[ \frac{d^-}{d \ell} g_R(\ell)  \ge \frac{ c_3 g_R(\ell) (1-g_R(\ell))}{\max_{S \in \mathcal{S}_1} \textrm{Rev}(S) }   \ge \frac{ c_4 g_R(\ell)(1-g_R(\ell))}{ \frac{1}{R} \sum_{i = 0}^{R-1} g_i(\ell)}  \]
for some $c_3, c_4 > 0$ and sufficiently large $R$, which gives \eqref{e:dcrt}.

\smallskip
We now argue that \eqref{e:dcrt} implies the result. First assume that there exists a $\ell_0 > \ell_c$ such that $g(\ell_0) < 1$ (this is clear if $f$ satisfies \eqref{a:pos2}, since then $\P [ \sup_{x \in \Lambda_{3r}} f(x) < -\ell ]    >  0$ for every $\ell \in \R$, but not in general). Then by monotonicity $1 - g_\ell(R) >   (1 - g(\ell_0))/2$ for all $\ell < \ell_0$ and large $R$. Hence setting $c_5 = c_1 (1-g(\ell_0))/2> 0$ and defining $f_R(\ell) = g_R(\ell)/c_5$ we have
\[   \frac{d^-}{d \ell} f_R(\ell)  \ge   \frac{f_R(\ell)}{ \frac{1}{R} \sum_{i = 0}^{R-1}  f_i(\ell) } . \]
for all $\ell < \ell_0$ and large $R$, and applying \cite[Lemma 3.1]{dcrt19a}\footnote{Although this lemma is stated for differentiable functions, it is easy to check that the proof goes through without differentiability since it only uses $f(b) - f(a) \ge \int_a^b   \frac{d^-}{d x} f(x) dx$.} yields the result. On the other hand, if $g(\ell_0) = 1$ for every $\ell_0 > \ell_c$ then the first statement of the theorem is immediate. To prove the second statement, instead choose a $\ell_0 < \ell_c$ and repeat the above argument. This implies the statement for $\ell < \ell_0$, and taking $\ell_0 \uparrow \ell_c$ gives the claim.
\end{proof}
\begin{remark}
 In the case of continuous $f$ the Russo-type inequality in Proposition \ref{p:russo} was crucial in the proof of Theorem \ref{t:expdecay}, although Proposition \ref{p:russodis} would be sufficient in the discrete case.
\end{remark}

\medskip

\appendix

\section{Orthogonal decomposition of $f_S$}

For completeness we present a classical orthogonal decomposition of the Gaussian field on $\X^d$, $\X \in \{\R, \Z\}$,
\[ f_S(\cdot)  = (q \star W|_S)(\cdot) = \int_{y \in S} q(\cdot-y) dW(y) , \]
where $S \subset \mathbb{T}^d$ is a compact domain, $q \in L^2(\mathbb{T}^d)$, and $W$ is the white noise on $\mathbb{T}^d$. In the case $\X = \R$ we also assume that $f_S$ is continuous.

\begin{proposition}[Orthogonal decomposition of $f_S$]
\label{p:odecom}
Let $(Z_i)_{i \ge 1}$ be a sequence of i.i.d.\ standard normal random variables and let $(\varphi_i)_{i \ge 1}$ be an orthonormal basis of $L^2(S)$. Then, as $n \to \infty$,
\[ f^n_S := \sum_{i \ge 1}^n Z_i (q \star \varphi_i) \Rightarrow f_S  \]
in law with respect to the $C^0(\X^d)$-topology on compact sets. In particular,
\[ f_S(\cdot) \stackrel{d}{=} \frac{Z_1 (q \star \id_S)(\cdot)}{\sqrt{\text{Vol}(S)}} + g(\cdot) \]
where $g$ is a Gaussian field independent of $Z_1$.
\end{proposition}

\begin{proof}
Consider the case $\X=\R$. Remark that, for each $x \in \mathbb{R}^d$, $f^n_S(x) \Rightarrow f_S(x)$ in law since they are centred Gaussian random variables and
\[  \mathbb{E} \Big[ \Big(\sum_{i \ge 1}^n Z_i (q \star \varphi_i)(x)  \Big)^2 \Big] = \sum^n_{i \ge 1} \Big(\int_S q(x-s) \varphi_i(s) \, dx \Big)^2  \to   \int_S q(x-s)^2  \, dx    = \mathbb{E}[f_S(x)^2]  \]
 by Parseval's identity. Note also that the functions $q \star \varphi_i$ are continuous (as a convolution of $L^2$ functions), and so each $f^n_S$ is continuous. Hence the first statement of the proposition follows by an application of Lemma \ref{l:in} below. For the second statement, set $\varphi_1$ to be constant on $S$. 
 
 The case $\X = \Z$ is similar but simpler; in fact $L^2(S)$ is finite-dimensional in that case, so $f_S^n  \stackrel{d}{=} f_S$ for sufficiently large $n$.
\end{proof}

\begin{lemma}
\label{l:in}
Let $(f_i)_{i \ge 1}$ be a sequence of independent continuous centred Gaussian fields on~$\mathbb{R}^d$ and define $g_n :=\sum^n_{i \ge 1} f_i$. Suppose there exists a continuous Gaussian field~$g$ on $\R^d$ such that, for each $x \in \mathbb{R}^d$, $g_n(x) \Rightarrow g(x)$ in law. Then $g_n \Rightarrow g$ in law with respect to the $C^0$-topology on compact sets.
\end{lemma}
\begin{proof}
We follow the proof of \cite[Theorem 3.1.2]{at07}. Since $g_n(x)$ is a sum of independent random variables converging in law, by Levy's equivalence theorem we may define $g(x)$ as the almost sure limit of $g_n(x)$. Fix a compact set $\Omega \subset \mathbb{R}^d$, and consider $(g_n)_{n \ge 1}$ as elements of the Banach space $C(\Omega)$ of continuous functions on $\Omega$ equipped with the $C_0$-topology. By the It\^{o}-Nisio theorem \cite[Theorem 3.1.3]{at07}, it suffices to show that 
\[  \int_\Omega  g_n \, d \mu \to \int_\Omega g \, d \mu \]
in mean (and so in probability) for every finite signed Borel measure $\mu$ on $\Omega$. Define the continuous functions $u_n(x) := \mathbb{E}[g_n(x)^2]$ and $u(x) := \mathbb{E}[g(x)^2]$. Then
\[ \mathbb{E}\Big[  \Big| \int_\Omega  g  \, d \mu - \int_\Omega  g_n \, d \mu \Big| \Big]  \le \int_\Omega \Big( \mathbb{E}\big[ \big( g(x) - g_n(x) \big)^2 \big] \Big)^{1/2}  \,  |\mu|(dx)  \le \int_\Omega \Big( u(x) - u_n(x) \Big)^{1/2} |\mu|(dx) .  \]
Since $u_n \to u$ monotonically, by Dini's theorem the convergence is uniform on $\Omega$, so we have that $ \mathbb{E}[  | \int_\Omega  g  \, d \mu - \int_\Omega  g_n \, d \mu | ] \to 0$ as required.
\end{proof}

\bigskip
\bibliographystyle{plain}
\bibliography{paper}

\begin{thebibliography}{10}

\bibitem{at07}
R.J. Adler and J.E. Taylor.
\newblock {\em Random fields and geometry}.
\newblock Springer, 2007.

\bibitem{ab87}
M.~Aizenman and D.J. Barsky.
\newblock Sharpness of the phase transition in percolation models.
\newblock {\em Comm. Math. Phys.}, 108(3):489--526, 1987.

\bibitem{an84}
M.~Aizenman and C.M. Newman.
\newblock Tree graph inequalities and critical behavior in percolation models.
\newblock {\em J. Stat. Phys.}, 36:107--143, 1984.

\bibitem{alex96}
K.S. Alexander.
\newblock Boundedness of level lines for two-dimensional random fields.
\newblock {\em Ann. Probab.}, 24(4):1653--1674, 1996.

\bibitem{bg17}
V.~Beffara and D.~Gayet.
\newblock Percolation of random nodal lines.
\newblock {\em Publ. Math. IHES}, 126(1):131--176, 2017.

\bibitem{bmm20}
D.~Beliaev, M.~McAuley, and S.~Muirhead.
\newblock Smoothness and monotonicity of the excursion set density of planar
  {G}aussian fields.
\newblock {\em Electron. J. Probab.}, 25(93):1--37, 2020.

\bibitem{bmm19}
D.~Beliaev, M.~McAuley, and S.~Muirhead.
\newblock Fluctuations in the number of excursion sets of planar {G}aussian
  fields.
\newblock {\em Probab. Math. Phys.}, 3(1), 2022.

\bibitem{bmr20}
D.~Beliaev, S.~Muirhead, and A.~Rivera.
\newblock A covariance formula for topological events of smooth {G}aussian
  fields.
\newblock {\em Ann. Probab.}, 48(6):2845--2893, 2020.

\bibitem{bsw05}
I.~Benjamini, O.~Schramm, and D.B. Wilson.
\newblock Balanced {B}oolean functions that can be evaluated so that every
  input bit is unlikely to be read.
\newblock In {\em STOC'05: Proceedings of the 37th Annual ACM Symposium on
  Theory of Computing}, pages 244--250, 2005.

\bibitem{bcks99}
C.~Borgs, J.T. Chayes, H.~Kesten, and J.~Spencer.
\newblock Uniform boundedness of critical crossing probabilities implies
  hyperscaling.
\newblock {\em Random Struct. Algo.}, 15(3--4):368--413, 1999.

\bibitem{ccfs86}
J.T. Chayes and L.~Chayes.
\newblock Finite-size scaling and correlation lengths for disordered systems.
\newblock {\em Phys. Rev. Lett.}, 57(24):2999--3002, 1986.

\bibitem{cc86}
J.T. Chayes and L.~Chayes.
\newblock Inequality for the infinite-cluster density in {B}ernoulli
  percolation.
\newblock {\em Phys. Rev. Lett.}, 56(16):1619--1622, 1986.

\bibitem{dg20}
V.~Dewan and D.~Gayet.
\newblock Random pseudometrics and applications.
\newblock {\em arXiv preprint arXiv:2004.05057}, 2020.

\bibitem{dm21b}
V.~Dewan and S.~Muirhead.
\newblock Mean field bounds for {P}oisson-{B}oolean percolation.
\newblock {\em arXiv preprint arXiv:2111.09031}, 2021.

\bibitem{dpr21}
A.~Drewitz, A.~Pr\'{e}vost, and P.-F. Rodriguez.
\newblock Critical exponents for a percolation model on transient graphs.
\newblock {\em arXiv preprint arXiv:2101.05801}, 2021.

\bibitem{dgrs20}
H.~Duminil-Copin, S.~Goswami, P.-F. Rodriguez, and F.~Severo.
\newblock Equality of critical parameters for percolation of {G}aussian free
  field level-sets.
\newblock {\em to appear in Duke. Math. J.}, 2020.

\bibitem{dcmt20}
H.~Duminil-Copin, I.~Manolescu, and V.~Tassion.
\newblock Planar random-cluster model: fractal properties of the critical
  phase.
\newblock {\em Probab. Theory Related Fields}, 181:401--449, 2021.

\bibitem{dcrt19b}
H.~Duminil-Copin, A.~Raoufi, and V.~Tassion.
\newblock Exponential decay of connection probabilities for subcritical
  {V}oronoi percolation in $\mathbb{R}^d$.
\newblock {\em Probab. Theory Related Fields}, 173(1--2):479--490, 2019.

\bibitem{dcrt19a}
H.~Duminil-Copin, A.~Raoufi, and V.~Tassion.
\newblock Sharp phase transition for the random-cluster and {P}otts models via
  decision trees.
\newblock {\em Ann. Math.}, 189(1):75--99, 2019.

\bibitem{dcrt20}
H.~Duminil-Copin, A.~Raoufi, and V.~Tassion.
\newblock Subcritical phase of $d$-dimensional {P}oisson-{B}oolean percolation
  and its vacant set.
\newblock {\em Ann. H. Lebesgue}, 3:677--700, 2020.

\bibitem{dct16}
H.~Duminil-Copin and V.~Tassion.
\newblock A new proof of the sharpness of the phase transition for {B}ernoulli
  percolation and the {I}sing model.
\newblock {\em Commm. Math. Phys}, 343:725--745, 2016.

\bibitem{egr04}
W.~Ehm, T.~Gneiting, and D.~Richards.
\newblock Convolution roots of radial positive definite function with compact
  support.
\newblock {\em Trans. Am. Math. Soc.}, 356(11):4655--4685, 2004.

\bibitem{fh17}
R.~Fitzner and R.~van~der Hofstad.
\newblock Mean-field behavior for nearest-neighbor percolation in $d > 10$.
\newblock {\em Electron. J. Probab.}, 22:65 pp., 2017.

\bibitem{gkr88}
A.~Gandolfi, M.~Keane, and L.~Russo.
\newblock On the uniqueness of the infinite occupied cluster in dependent
  two-dimensional site percolation.
\newblock {\em Ann. Probab.}, 16(3):1147--1157, 1988.

\bibitem{gps10}
C.~Garban, G.~Pete, and O.~Schramm.
\newblock The {F}ourier spectrum of critical percolation.
\newblock {\em Acta Math.}, 205(1):19--104, 2010.

\bibitem{gv19}
C.~Garban and H.~Vanneuville.
\newblock Bargmann-{F}ock percolation is noise sensitive.
\newblock {\em Electron. J. Probab.}, 25:1--20, 2020.

\bibitem{grs22}
S.~Goswami, P.-F. Rodriguez, and F.~Severo.
\newblock On the radius of {G}aussian free field excursion clusters.
\newblock {\em to appear in Ann. Probab.}, 2022.

\bibitem{gri99}
G.R. Grimmett.
\newblock {\em Percolation}.
\newblock Springer, 1999.

\bibitem{ham57}
J.M. Hammersley.
\newblock Percolation processes: {L}ower bounds for the critical probability.
\newblock {\em Ann. Math. Statist.}, 28:790--795, 1957.

\bibitem{har90}
T.~Hara.
\newblock Mean-field critical behaviour for correlation length for percolation
  in high dimensions.
\newblock {\em Probab. Theory Related Fields}, 86:337--385, 1990.

\bibitem{har08}
T.~Hara.
\newblock Decay of correlations in nearest-neighbor self-avoiding walk,
  percolation, lattice trees and animals.
\newblock {\em Ann. Probab.}, 36(2):530--593, 2008.

\bibitem{har60}
T.E. Harris.
\newblock A lower bound for the critical probability in a certain percolation
  process.
\newblock {\em Proc. Camb. Phil. Soc.}, 56:13--20, 1960.

\bibitem{kes82}
H.~Kesten.
\newblock {\em Percolation theory for mathematicians}.
\newblock Progress in Probability and Statistics Vol. 2. Springer, 1982.

\bibitem{kes87}
H.~Kesten.
\newblock Scaling relations for $2${D}-percolation.
\newblock {\em Comm. Math. Phys}, 109:109--156, 1987.

\bibitem{kt20}
L.~K\"{o}hler-Schindler and V.~Tassion.
\newblock Crossing probabilities for planar percolation.
\newblock {\em arXiv preprint arXiv:2011.04618}, 2020.

\bibitem{kn11}
G.~Kozma and A.~Nachmias.
\newblock Arm exponents in high dimensional percolation.
\newblock {\em J. Amer. Math. Soc.}, 24(2):375--409, 2011.

\bibitem{kul78}
S.~Kullback.
\newblock {\em Information theory and statistics}.
\newblock Dover, 1978.

\bibitem{led98}
M.~Ledoux.
\newblock A short proof of the {G}aussian isoperimetric inequality.
\newblock In E.~Eberlein, M.~Hahn, and M.~Talagrand, editors, {\em High
  Dimensional Probability. Progress in Probability, vol 43.}, pages 229--232.
  Birkh\"{a}user, Basel, 1998.

\bibitem{men86}
M.~Menshikov.
\newblock Coincidence of critical points in percolation problems.
\newblock {\em Sov. Math. Dokl.}, 33:856--859, 1986.

\bibitem{ms83a}
S.A. Molchanov and A.K. Stepanov.
\newblock Percolation in random fields. {I}.
\newblock {\em Theor. Math. Phys.}, 55(2):478--484, 1983.

\bibitem{ms83b}
S.A. Molchanov and A.K. Stepanov.
\newblock Percolation in random fields. {II}.
\newblock {\em Theor. Math. Phys.}, 55(3):592--599, 1983.

\bibitem{mrv20}
S.~Muirhead, A.~Rivera, and H.~Vanneuville (with an appendix~by
  L.~K\"{o}hler-Schindler).
\newblock The phase transition for planar {G}aussian percolation models without
  {FKG}.
\newblock {\em arXiv preprint arXiv:2010.11770}, 2020.

\bibitem{mv20}
S.~Muirhead and H.~Vanneuville.
\newblock The sharp phase transition for level set percolation of smooth planar
  {G}aussian fields.
\newblock {\em Ann. I. Henri Poincar\'e Probab. Stat.}, 56(2):1358--1390, 2020.

\bibitem{osss05}
R.~O'Donnell, M.~Saks, O.~Schramm, and R.A. Servedio.
\newblock Every decision tree has an influential variable.
\newblock In {\em 46th Annual IEEE Symposium on Foundations of Computer Science
  (FOCS'05)}, pages 31--39, 2005.

\bibitem{os07}
R.~O'Donnell and R.A. Servedio.
\newblock Learning monotone decision trees in polynomial time.
\newblock {\em SIAM J. Comput.}, 37(3):827--844, 2007.

\bibitem{pi82}
L.D. Pitt.
\newblock Positively correlated normal variables are associated.
\newblock {\em Ann. Probab.}, 10(2):496--499, 1982.

\bibitem{riv19}
A.~Rivera.
\newblock Talagrand's inequality in planar {G}aussian field percolation.
\newblock {\em Electron J. Probab.}, 26:1--25, 2021.

\bibitem{rv19}
A.~Rivera and H.~Vanneuville.
\newblock Quasi-independence for nodal lines.
\newblock {\em Ann. Inst. H. Poincar{\'{e}} Probab. Statist.},
  55(3):1679--1711, 2019.

\bibitem{rv20}
A.~Rivera and H.~Vanneuville.
\newblock The critical threshold for {B}argmann-{F}ock percolation.
\newblock {\em Ann. Henri Lebesgue}, 3, 2020.

\bibitem{rod17}
P.-F. Rodriguez.
\newblock A $0$-$1$ law for the massive {G}aussian free field.
\newblock {\em Probab. Theory Related Fields}, 169(3--4), 2017.

\bibitem{rud70}
W.~Rudin.
\newblock An extension theorem for positive-definite functions.
\newblock {\em Duke Math. J.}, 37(1):49--53, 1970.

\bibitem{ss11}
O.~Schramm and S.~Smirnov (with an appendix~by C.~Garban).
\newblock On the scaling limits of planar percolation.
\newblock {\em Ann. Probab.}, 39(5):1768--1814, 2011.

\bibitem{sw01}
S.~Smirnov and W.~Werner.
\newblock Critical exponents for two-dimensional percolation.
\newblock {\em Math. Res. Lett.}, 8(5):729--744, 2001.

\bibitem{tas87}
H.~Tasaki.
\newblock Hyperscaling inequalities for percolation.
\newblock {\em Comm. Math. Phys}, 113(1):49--65, 1987.

\bibitem{bd20}
R.~van~den Berg and H.~Don.
\newblock A lower bound for point-to-point connection probabilities in critical
  percolation.
\newblock {\em Electron. Comm. Probab.}, 47, 2020.

\bibitem{vn20}
R.~van~den Berg and P.~Nolin.
\newblock On the four-arm exponent for $2d$ percolation at criticality.
\newblock In M.E. Vares, R.~Fern{\'a}ndez, L.R. Fontes, and C.M. Newman,
  editors, {\em In and Out of Equilibrium 3: Celebrating Vladas Sidoravicius,
  Progress in Probability Vol. 77.}, pages 125--145. Springer, 2021.

\bibitem{wei84}
A.~Weinrib.
\newblock Long-range correlated percolation.
\newblock {\em Phys. Rev. B}, 29(1):387, 1984.

\end{thebibliography}

\end{document}